\documentclass{elsarticle}

\usepackage{amsmath}
\usepackage{amssymb}
\usepackage{amsthm}
\usepackage{hyperref}
\usepackage{esint}
\usepackage{mathrsfs}
\usepackage{tikz}
\usepackage{dsfont}
\usetikzlibrary{matrix,arrows,decorations.pathmorphing}
\usepackage{setspace}
\usepackage{epigraph}

\newcommand{\eps}{\epsilon}
\newcommand{\ep}{\varepsilon}
\newcommand{\LL}{\langle}
\newcommand{\RR}{\rangle}
\newcommand{\pd}{\partial}

\newcommand{\Ex}{\mathbb{E}}

\newcommand{\vp}{\varphi}
\newcommand{\sgn}{\mathrm{sgn}}

\newcommand{\x}{\mathbf{x}}
\newcommand{\y}{\mathbf{y}}
\newcommand{\z}{\mathbf{z}}
\newcommand{\h}{\mathbf{h}}

\newcommand{\A}{\mathbf{A}}
\newcommand{\B}{\mathbf{B}}
\newcommand{\FF}{\mathbf{F}}
\newcommand{\GG}{\mathbf{G}}

\newcommand{\ab}{{\boldsymbol\alpha}}
\newcommand{\be}{{\boldsymbol\beta}}

\renewcommand{\d}{{\rm d}}

\newcommand{\abs}[1]{\left | #1 \right |}
\newcommand{\norm}[1]{\left \| #1 \right \|}
\newcommand{\bk}[1]{\left( #1 \right)}

\newcommand{\R}{\mathbb{R}}
\newcommand{\T}{\mathbb{T}}
\newcommand{\Td}{\mathbb{T}^d}

\newtheorem{thm}{Theorem}[section]
\newtheorem{lma}{Lemma}[section]
\newtheorem{prop}{Proposition}[section]
\newtheorem{cor}{Corollary}[section]
\newdefinition{rem}{Remark}[section]
\newdefinition{defin}{Definition}[section]
\newproof{pf}{Proof}

\numberwithin{equation}{section}
\numberwithin{figure}{section}

\begin{document}

\begin{frontmatter}
\title{Nonlinear Anisotropic Degenerate Parabolic-Hyperbolic Equations with Stochastic Forcing}

\author[1]{Gui-Qiang G. Chen\corref{cor1}%
\fnref{fn1}}\ead{chengq@maths.ox.ac.uk}

\author[1,2]{Peter H.C. Pang\fnref{fn2}}
\ead{peter.pang@ntnu.no}

\cortext[cor1]{Corresponding author}
\fntext[fn1]{The research of Gui-Qiang G. Chen  was supported in part by
the UK
Engineering and Physical Sciences Research Council Awards
EP/L015811/1 and EP/V008854/1, and the Royal Society--Wolfson Research Merit Award WM090014 (UK).
}
\fntext[fn2]{The research of Peter Pang
was supported in part by the UK EPSRC Science and Innovation Award
to the Oxford Centre for Nonlinear PDE (EP/E035027/1),
a Croucher Oxford Scholarship granted by the Croucher Foundation,
and the Research Council of Norway Toppforsk Project on Waves and Nonlinear Phenomena (250070).
}

\address[1]{Mathematical Institute, University of Oxford,
Oxford, OX2 6GG, UK}
\address[2]{Department of Mathematical Sciences,
Norwegian University of Science and Technology,
NO-7491 Trondheim,
Norway}


\begin{abstract}
We are concerned with nonlinear anisotropic degenerate parabolic-hyperbolic equations with stochastic forcing,
which are heterogeneous ({\it i.e.}, not space-translational invariant).
A unified framework is established for the continuous dependence estimates,
fractional $BV$
regularity estimates,
and well-posedness for stochastic kinetic solutions of
the nonlinear stochastic degenerate parabolic-hyperbolic equation.
In particular, we establish
the well-posedness of the nonlinear stochastic equation
in $L^p \cap N^{\kappa,1}$ for $p\in (1,\infty)$ and the $\kappa$--Nikolskii space $N^{\kappa,1}$ with $\kappa>0$,
and the $L^1$--continuous dependence of the stochastic kinetic solutions not only on the initial data, but also on
the degenerate diffusion matrix function, the flux function, and the multiplicative noise function involved in the nonlinear
equation.
\end{abstract}

\end{frontmatter}

\section{Introduction}\label{sec:I_intro}

We are concerned with the continuous dependence
of stochastic kinetic solutions
of the Cauchy problem for the nonlinear anisotropic degenerate parabolic-hyperbolic equations with stochastic forcing:
\begin{align}
\pd_t u  + \nabla \cdot \FF(u,\x) =\nabla\cdot(\A(u)\nabla u)+ \sigma(u) \dot{W}  \,\qquad \mbox{for $\x\in \T^d$},
\label{eq:eq1_parab}
\end{align}
and initial data:
\begin{align}
u|_{t=0}=u_0(\x),
\label{eq:1.2}
\end{align}
where $\A(u)$ is a positive semi-definite matrix function so that there exists a positive semi-definite matrix
$\ab$ with $\A(u)=\ab(u)\ab(u)^\top$,
the flux function $\FF(u,\x)=(F^1,F^2,\cdots, F^d)(u,\x)$ is heterogeneous (depending on the space variable $\x$),
and $\sigma(u)$ is a multiplicative noise function.
In the noise term, $W=W(t)$ is a standard (one-dimensional) Brownian motion
on the abstract stochastic basis $(\Omega, \mathcal{F},\{\mathcal{F}_t\}_{t \ge 0},\mathbb{P})$.

In this paper, we first develop a unified framework for the continuous dependence estimates
on not only the initial data $u_0(\x)$ but also the
diffusion matrix $\A(u)$, the
flux function $\FF(u,\x)$, and the multiplicative noise function $\sigma(u)$.
Then we derive from this continuous dependence framework to obtain
both an $L^1$--stability property
and a fractional $BV$ estimate,
{\it i.e.}, a Nikolskii semi-norm estimate defined by \eqref{1.2b} below,
for stochastic kinetic solutions.
The motivation for such a study is three-fold:
First, equation \eqref{eq:eq1_parab} is heterogeneous ({\it i.e.}, not space-translational invariant)
so that the $BV$-in-space estimate
of solutions in terms of the $BV$ initial data
does not follow directly from the $L^1$--stability of solutions,
which is different from the space-translational invariant case as treated in Chen--Ding--Karlsen \cite{CDK2012}.
In fact, the $BV$-in-space estimate can be obtained only in the special case
that $D_\x\cdot\FF(u,\x):=\sum_{j=1}^d F_{x_j}^j(u,\x)$ is Lipschitz in its spatial argument $\x$;
in general, only a fractional $BV$-in-space bound ({\it i.e.}, bounded in the Nikolskii semi-norm) can be obtained,
which depends on the H\"older norm of $D_\x\cdot\FF(u,\x)$ in $\x$, as observed in this paper.
Second, we carry out our analysis  directly from the definition of stochastic kinetic solutions, which is independent of the choices
of approximate solutions,
different from \cite{CDK2012}.
Most importantly, we provide a uniform treatment for the $L^1$--continuous dependence estimates,
fractional $BV$ regularity estimates, and
well-posedness for stochastic kinetic solutions.
For the deterministic case, similar stability problems have been analyzed; see \cite{CK2006,Lec2011}
and the references cited therein.

For nonlinear stochastic hyperbolic balance laws:
\begin{align}
\pd_t u + \nabla \cdot \FF(u) = \sigma(u) \dot{W},
\label{eq:_conservation}
\end{align}
the $L^1$--continuous dependence estimates on the flux function $\FF(u)$,
the noise function $\sigma(u)$, and the initial data
$u_0(\x)$ have been established in Chen--Ding--Karlsen \cite{CDK2012},
based on the earlier work of Feng--Nualart \cite{FN2008} on the well-posedness
for \eqref{eq:_conservation}.
In \cite{FN2008}, the existence of strong stochastic entropy solutions,
which involve a non-adapted stochastic integral,
is achieved by the compensated compactness framework in Chen--Lu \cite{CL1989} for
$d = 1$.
In Chen--Ding--Karlsen \cite{CDK2012}, this restriction ($d=1$) is first removed
by combining the $BV$-estimate they developed with the $L^1$--contraction estimate of the
$BV$ solutions
so that the multidimensional case $d\ge 2$ can be handled.
It is observed in Karlsen--Storr\o sten \cite{KS2017} that there are other ways to achieve the well-posedness
to capture the noise-noise interaction in the comparison between two solutions,
without resorting to the {\it prima facie} contrived notion of strong stochastic entropy solutions
proposed in \cite{FN2008}.
One of them is the kinetic formulation approach that has been carried out
in Debussche--Vovelle \cite{DV2010} for \eqref{eq:_conservation},
in which the notion of strong entropy solutions can be avoided via introducing the kinetic defect measure;
in this approach, by linearizing the equation via the introduction of a new kinetic variable,
the interaction in certain cross terms involving the noise can
be handled by the use of the defect measure, instead of the direct integration
(see also \cite{DV2018}).
In Bauzet--Vallet--Wittbold \cite{BVW2012}, the formulation of strong entropy solutions is avoided
by comparing the stochastic entropy solution
directly to the corresponding vanishing viscosity solution.
Furthermore, in \cite{KS2017},  the Kruzhkov entropy condition is modified to
compare a solution
to a general Malliavin differentiable variable (instead of a constant),
by using an anticipating It\^o formula;
the vanishing viscosity solution is shown to be Malliavin differentiable,
and the framework in \cite{BVW2012} is used, for which indicates where the notion of strong stochastic entropy solutions
in \cite{FN2008} may arise ({\it cf.} Remark 5.1 in \cite{KS2017}).
It would be interesting to study underlying theoretical connections
between the kinetic formulation approach and the Malliavin calculus approach.

It bears pointing out that, in the very specific context of continuous dependence for \eqref{eq:eq1_parab},
the deterministic and stochastic theories diverge,
and we encounter difficulties and structures peculiar to stochastic balance laws.
In particular, the It\^o correction difference may prevent an account of continuous dependence
with the forcing terms involving on the solution itself, in addition to the spatial or temporal variables.
However, it is still possible to consider the problem as we do here for which the flux depends on the spatial variable directly,
which may have applications in considering stochastic balance laws
on manifolds \cite{GK2019},
where a connection is spatially dependent,
and the kinetic formulation of the equation is more intricate.

Other variations on the well-posedness theory of balance laws have been considered.
The most prominent of these arise with conservative Stratonovich noises,
in which the noise takes the divergence form.
These are of some interest in physical systems as they arise from the perturbation
of characteristics \cite{FGP2010}.
In particular, Fehrman--Gess \cite{FG2019} investigated the well-posedness and
continuous dependence of the stochastic degenerate
parabolic equation of porous medium type:
$$
\pd_t u + \nabla A(\x,u)\circ \d z_t = \Delta (|u|^{m-1}u),
$$
including the fast diffusion case $m<1$, where $z_t$ is a geometric rough path,
which includes the case that $z_t$ is a finite-dimensional Brownian motion.
This builds on the results collected in \cite{BDR2016} for the stochastic PDEs of this form.
Also see \cite{CP2019,DV2010} for the existence of invariant measures
for nonlinear conservation laws driven by stochastic forcing.

\smallskip
This paper consists of eight sections.
In \S \ref{sec:I_kinetic_formulation}, we introduce the notion of stochastic kinetic solutions in a {\it divergence form}
for \eqref{eq:eq1_parab}.
In \S \ref{sec:I_preparation}, we develop a general framework for the continuous dependence estimates
of the stochastic kinetic solutions.
In \S \ref{sec:I_L1_estimates}, we employ the framework in \S \ref{sec:I_preparation}
to establish the $L^1$--stability of stochastic kinetic solutions of equation \eqref{eq:eq1_parab}.
In \S \ref{sec:I_BV_estimate}, we employ the framework
to derive the fractional $BV$ estimate ({\it i.e.}, the Nikolskii semi-norm estimate).
Using the fractional $BV$ estimate
in \S \ref{sec:I_BV_estimate},
we complete the $L^1$--continuous dependence estimate in \S\ref{sec:I_continuous_dependence}.
In \S\ref{sec:I_existence}, we establish the existence of stochastic kinetic solutions.
In \S 8, we derive a temporal fractional $BV$ estimate of stochastic kinetic solutions.

\smallskip
Before we proceed further, we address two notational points:
First, we denote $\nabla$ the material derivative,
and $\nabla^i$ the material derivative in the $x_i$--variable
(the $i$th coordinate of $\nabla$)
so that
\begin{align*}
&\nabla^i \FF(u,\x) := \FF_u (u,\x) \nabla^i u + \FF_{x_i}(u,\x), \\
& D_\x\cdot\FF(\cdot, \x):=F_{x_i}^i(\cdot,\x),\\
&\nabla \cdot \FF(u,\x) := F^i_u(u,\x)\nabla^i u + D_\x\cdot \FF(u,\x),
\end{align*}
where we have used the Einstein summation convention that repeated indices are implicitly summed over,
which will be used throughout this paper from now on.
Second, the Nikolskii space $N^{\kappa,1}$ is defined by
\begin{equation}\label{1.2a}
v \in N^{\kappa,1} \quad \longleftrightarrow \quad \|v\|_{N^{\kappa,1}}:=\Ex\big[\|v\|_{L^1}\big] + \Ex\big[|v|_{N^{\kappa,1}}\big] < \infty,
\end{equation}
which forms a Banach space ({\it cf}. \cite{Sim1990} for the deterministic case), where the semi-norm $\Ex\big[|v|_{N^{\kappa,1}}\big]$ is defined by
\begin{equation}\label{1.2b}
\Ex\big[|v|_{N^{\kappa,1}}\big]:= \sup_{|\h| > 0} \Ex\big[\int \frac{|v(\y + \h) - v(\y)|}{|\h|^{\kappa }} \;\d\y\big].
\end{equation}

We assume that the functions involved satisfy the following conditions for $\x,\y\in \T^d$:
\begin{align}
&D_\x\cdot\FF_u(\cdot,\x)= F^j_{ux_j}(\cdot,\x) \in L^\infty, \label{eq:A_F}\\
&|\FF_u(u,\x) - \FF_u(v,\x)|\leq
C\big(|u|^{p- 1} + |v|^{p - 1}+1\big) |u - v|^{\kappa_{F1}}, \label{eq:A_kappa1}\\
&|D_\x\cdot \FF(u,\x)  - D_\y\cdot \FF(u, \y)| \leq C\big(|u|^{q}+1\big)|\x - \y|^{\kappa_{F2}}, \label{eq:A_kappa2}\\
&|\sigma(u) - \sigma(v)| \leq C|u - v|^{\lambda_\sigma}, \label{eq:A_sigma}\\
&\sup_{i,j}|\ab_{ij}(u) - \ab_{ij}(v)| \leq C|u -v|^{\gamma_\ab} \label{eq:A_alpha}
\end{align}
for some constant $C>0$,
where
$\FF_u(u,\cdot)$ and $\FF_{x_j}(\cdot,\x)$ denote the partial derivatives with respect to $u$ and $x_j$ respectively,
$\kappa_{F1} > 0$, $\kappa_{F2} > 0$, $\lambda_\sigma >\frac{1}{2}$, and  $\gamma_\ab > \frac{1}{2}$.
We also assume that $D_\x\cdot \FF(u,\x)$ and $\sigma(u)$ have at most linear growth in $u$,
and $\A(u)$ has polynomial growth in $u$.

\smallskip
The results established in this paper on $\T^d$ can directly be extended to the whole space $\R^d$ by the
techniques developed here. For this purpose, it requires to modify the test function in the proof arguments
by multiplying a non-negative smooth weight function
with appropriate decay rate at infinity.
The results established here can also be extended to more general stochastic forcing
such as a multidimensional or a cylindrical Brownian motion:
$$
 \d B = \d B(u,t)=\sum_{k=0}^{m} \LL\Phi(u) , \d W_k(t) \mathbf{e}_k\RR_\mathcal{H},
$$
where $\mathcal{H}$ is an $m$-dimensional Hilbert space (with $m$ possibly infinite)
with a complete orthonormal basis $\{\mathbf{e}_k\}$, $W_k$ are the independent standard Brownian motions,
and $\Phi: \mathbb{R} \to \mathcal{H}$
with $\LL\Phi(u), \mathbf{e}_k\RR_\mathcal{H} = g_k(u)$ and $\sum_k g_k^2(u) \leq C(|u|^2+1)$.
The results can also be adapted to the additive noise:
$$
\d B(\x,t)=\sum_{k=0}^{\infty} g_k(\x)\d W_k(t),
$$
where $\sum_k g_k^2 \in L^1(\Td)$.
It would be interesting to
extend our analysis
to the noises with all three arguments of form: $B(u,\x,t)=\sum_{k=0}^{\infty} g_k(u,\x)\d W_k(t)$.
There are new difficulties when the noises depend on both solution $u$ and the spatial variable $\x$
in doubling spatial variables in order to quantify the continuous dependence.
Essentially, one necessarily comes across the terms where the continuity of $g$ in the two arguments are in competition.
This competition manifests itself in the expressions such as $|g(\zeta, \y) - g(\xi,\x)|$
under an appropriate integral (see \eqref{3.13e} and \eqref{eq:correction_cde} below)
so that, if $\zeta\to \xi$ first, some terms become unbounded and, if $\y\to \x$ first, the other terms become unbounded.

\section{Stochastic Kinetic Formulation}\label{sec:I_kinetic_formulation}

In this section, we introduce the notion of stochastic kinetic solutions for \eqref{eq:eq1_parab},
motivated by the earlier work in  Chen--Perthame \cite{CP2003}; see also  Lions--Perthame--Tadmor \cite{LPT1994a}
for the hyperbolic case, and Debussche--Hofmanova--Vovelle \cite{DHV2016} and Gess--Souganidis \cite{GS2017}
for the translation-invariant degenerate parabolic treatment.
Because of the heterogeneity of the flux function $\FF=\FF(u,\x)$,
the definition of a stochastic kinetic solution has to be generalized
to preserve a structure of {\it divergence form};
see Definition \ref{def:ksolution} below.

\smallskip
We now motivate the notion of {\it stochastic kinetic solutions} heuristically as a
form of weak solutions.
Denote the Heaviside function $H(r) = \mathds{1}_{r > 0}(r)$.
Starting from the smooth approximate solutions $u^\eps$ satisfying the following equation with viscosity:
\begin{align}\label{eq:viscosity_approximation}
\pd_t u^\eps + \nabla \cdot \FF(u^\eps,\x)
= \nabla\cdot \big(\A(u^\eps)\nabla u^\eps\big)+ \sigma(u^\eps) \dot{W}  + \eps \Delta u^\eps,
\end{align}
we multiply both sides of \eqref{eq:viscosity_approximation} by $-H'(\xi - u^\eps)$ for the approximate solution $u^\eps$
to obtain
\begin{align}
&\pd_t H(\xi - u^\eps)\nonumber\\
&= \, H'(\xi - u^\eps)  \FF_u(u^\eps,\x) \cdot \nabla u^\eps + H'(\xi - u^\eps) D_\x\cdot \FF(u^\eps,\x) \nonumber\\
&\quad - \nabla \cdot \big(H'(\xi - u^\eps) \A(u^\eps) \nabla u^\eps\big)
+ \A(u^\eps) :  \big(\nabla H'(\xi - u^\eps)\otimes \nabla u^\eps\big)\nonumber\\
&\quad - H'(\xi - u^\eps) \sigma(u^\eps) \dot{W} + \frac{1}{2} H''(\xi - u^\eps) \sigma^2(u^\eps)\nonumber\\
&\quad + \eps \Delta H(\xi - u^\eps) + \eps \nabla H'(\xi - u^\eps)\cdot \nabla u^\eps\nonumber\\
&= - \nabla \cdot \big(\FF_u(\xi,\x)H(\xi - u^\eps)\big) + \pd_\xi\big(H(\xi - u^\eps) D_\x \cdot \FF(\xi,\x)\big) \nonumber\\
&\quad+ \A(\xi) : \nabla^2 H(\xi - u^\eps) - \delta(\xi - u^\eps) \sigma(\xi) \dot{W}\nonumber\\
&\quad- \pd_\xi \big(\eps \delta(\xi - u^\eps) |\nabla u^\eps|^2
 + \delta(\xi - u^\eps) \A(\xi) : (\nabla u^\eps\otimes \nabla u^\eps) - \frac{1}{2}\delta(\xi - u^\eps) \sigma^2(\xi)\big)\nonumber\\
&\quad + \eps\Delta H(\xi - u^\eps), \label{2.2a}
\end{align}
where we have used $H'(\xi - u^\eps) = \delta( \xi - u^\eps)$
and the colon to denote the element-wise scalar product so that
$\A:\B = \sum_{1 \leq i,j\leq d} a_{ij}b_{ij}$ for $d \times d$ matrices $\A =(a_{ij})$ and $\B = (b_{ij})$.

\medskip
Assume that $u^\eps(\x,t)\to u(\x,t)$ {\it a.e.} as $\eps\to 0$.
Then, letting $\eps\to 0$, we arrive at the \emph{kinetic formulation} of the equation:
\begin{align}
&\pd_t H(\xi - u) + \nabla \cdot \big(\FF_u(\xi,\x)  H(\xi - u) \big)
   -\pd_\xi \big(H(\xi - u) D_\x\cdot\FF(\xi,\x) \big) \nonumber \\
&=\A(\xi) : \nabla ^2 H(\xi - u) - \pd_\xi H(\xi - u) \sigma(\xi) \dot{W} - \pd_\xi(m^u + n^u - p^u),
\label{eq:eq2_kinetic}
\end{align}
where measures $m^u$, $n^u$, and $p^u=\sigma(\xi)\delta(\xi-u)\sigma(\xi)$ are the weak limits of the kinetic dissipation,
parabolic defect, and It\^o correction measures as $\eps \to 0$, respectively:
\begin{align*}
&\eps |\nabla u^\eps|^2 \delta(\xi - u^\eps)
\,{\rightharpoonup} \,m^u,\\
& \A(\xi) : (\nabla u^\eps \otimes \nabla u^\eps)\,\delta(\xi - u^\eps)
\,{\rightharpoonup} \,n^u,\\
&\frac{1}{2} \sigma^2(\xi)\delta(\xi - u^\eps)
\,{\rightharpoonup} \,p^u.
\end{align*}

Denote by $\mathfrak{M}_1(\mathbb{R})$ the set of probability measures on $\mathbb{R}$
and by $\mathfrak{M}_b^+$  the set of non-negative bounded Radon measures.
Moreover, denote $C^\infty_c$ the space of compactly supported smooth functions.
Let $\mathcal{L}_\R$ and $\mathcal{L}_{\T^d}$ respectively be the Lebesgue measure on $\R$ and on the flat torus $\T^d$.
Let $\mathcal{B}([0,T])$ be the Borel algebra on $[0,T]$ and let $\mathcal{B}(\mathbb{T}^d)$ be the Borel algebra on $\T^d$.
Let $\mathcal{P}_T$ be the predictable $\sigma$-algebra of $\mathcal{B}([0,T]) \otimes \mathcal{F}$;
that is, $\mathcal{P}_T$ is generated by all real-valued left-continuous processes adapted
to filtration $\{\mathcal{F}_t\}_{t\ge 0}$.
The predictable subspace $L^p_P(\Omega \times [0,T] \times \T^d)$ denotes the subspace of
functions $\mathbb{P}\otimes \mathcal{L}_{\R} \otimes \mathcal{L}_{\mathbb{T}^d}$--almost everywhere equal
to a $\mathcal{P}_T\otimes \mathcal{B}(\T^d)$-measurable function
of $L^p(\Omega \times [0,T]\times \T^d)$  (also see \cite[\S 2.1.1]{DV2018} and the references cited there).

We can now make the following definition, clarifying the roles of the measures exhibited above.

\begin{defin}[Stochastic kinetic solutions]\label{def:ksolution}
A function
$$
u \in L^p_P(\Omega \times [0,T]\times \T^d) \cap L^p(\Omega; L^\infty([0,T]; L^p(\T^d)))
$$
is called a \emph{stochastic kinetic solution} of (\ref{eq:eq1_parab})--\eqref{eq:1.2}
in $\Omega\times\T^d\times [0,T]$ for some $T>0$
provided that
$u$ satisfies the following conditions:

\smallskip
\begin{itemize}
\item[(i)]
$\nabla \cdot \int_0^u \ab(\xi) \;\d \xi \in L^2(\Omega \times \T^d\times [0,T])$;

\smallskip
\item[(ii)] For any $\varphi \in C_b(\mathbb{R})$ (bounded continuous functions),
the Chen--Perthame chain rule relation holds (see \cite{CP2003}):
\begin{align}\label{eq:art_chainrule}
\nabla \cdot \big(\int_0^u  \ab(\xi) \varphi(\xi) \;\d \xi\big)
= \varphi(u)\, \nabla \cdot \big(\int_0^u \ab(\xi) \;\d \xi\big)
\end{align}
in the sense of distributions in $\T^d$ and almost everywhere in $(\omega, t)${\rm ;}

\smallskip
\item[(iii)]
For any $\varphi \in C_c^1(\R \times \T^d)$, $t \mapsto \iint H(\xi - u(\x, t))\varphi(\xi, \x) \;\d \xi\,\d \x$
is c\`adl\`ag ({\it i.e.}, right-continuous with left limits);

\smallskip
\item[(iv)]
There are non-negative $\mathfrak{M}_b^+(\R\times \Td \times [0,T] )$--valued variables $m^u$, $n^u$, and $p^u$ such that
\begin{equation} \label{eq:definsol}
\begin{aligned}
&\int_0^T \iint  H(\xi - u) \pd_t \varphi \;\d \xi\, \d \x \,\d t +\int_0^T \iint  \FF_u(\xi,\x)\cdot \nabla \varphi \;\d\xi \,\d\x \,\d t\\
&\quad  \,\,-  \int_0^T \iint  D_\x\cdot\FF(\xi,\x) H(\xi - u)\,\varphi_\xi \;\d\xi \,\d\x \,\d t\\
&=-\int_0^T \iint H(\xi - u)   \A(\xi):\nabla ^2\varphi \;\d \xi\,\d\x \,\d t \\
&\quad\,  -\int_0^T \iint \varphi_\xi \,\d (m^u+n^u-p^u)(\xi, \x, t)\\
&\quad \, - \int_0^T \int \sigma(u)\varphi(u,\x,t) \;\d \x \,\d W(t)\
  + \iint  H(\xi - u_0)\varphi(\xi,\x,0)\;\d\xi \,\d\x
\end{aligned}
\end{equation}
almost surely, for any $\varphi \in C^\infty_c(\mathbb{R}, \T^d \times [0,T))$.
Here,
$p^u: \Omega \to \mathfrak{M}^+_b(\mathbb{R}\times \T^d \times \R_+)$
is an \emph{It\^o correction measure}:
\begin{equation} \label{eq:Itomeasure}
\,\,\, p^u(\varphi):= \frac{1}{2}\int_0^\infty \int_{\T^d} \sigma^2(u) \varphi(u,\x,t)\,\d\x\,\d t
\quad \mbox{for any $\varphi \in  {C_c}(\mathbb{R}\times \T^d \times \R_+)$,}
\end{equation}
$n^u : \Omega \to \mathfrak{M}^+_b(\mathbb{R}\times \T^d \times \R_+)$ is the
a \emph{parabolic defect measure}:
\begin{align} \label{eq:pdefectmeas}
n^u(\varphi):= \int_0^\infty\int_{\T^d}
  \Big|\nabla\cdot \big(\int_0^{u(\x,t)}\boldsymbol{\alpha}(\zeta)\;\d \zeta\big)\Big|^2 \varphi(u(\x,t),\x,t) \;\d\x\,\d t
\end{align}
{for any $\varphi \in {C_c}(\mathbb{R}\times \T^d \times \R_+)$,}
and  $m^u: \Omega \to \mathfrak{M}^+_b(\mathbb{R}\times \T^d \times \R_+)$ is the {\it kinetic defect measure}
satisfying that, for
any $\varphi \in  {C_c}(\mathbb{R}\times\T^d)$ and $t\in (0,T]$,
\[
\int_{\mathbb{R}\times \T^d \times [0,t]} \varphi(\xi, \x)\, \d_{(\xi,\x,s)} m^u(\xi, \x, s;\omega) \in L^2(\Omega\times [0,T])
\]
has predictable
representative (that is $\mathbb{P}\otimes \mathcal{L}_\R$-almost everywhere equal to a function in $L^2_P(\Omega \times [0,T])$).
\end{itemize}
\end{defin}

\medskip
\begin{rem}
In this section, we introduce the kinetic formulation \eqref{eq:eq2_kinetic}
for stochastic kinetic solutions in the sense of \eqref{eq:definsol}
with the associated kinetic measure $m^u$, parabolic defect measure $n^u$, and It\^o correction measure $p^u$ in the periodic domain.
The existence of stochastic kinetic solutions in the periodic domain will be established in \S 7.
The kinetic formulation can also be defined in $\R^d$ or any other domain, correspondingly.
Equation \eqref{eq:definsol} is obtained by testing (\ref{eq:eq2_kinetic})
with $\varphi$ and using the Chen--Perthame chain rule (\ref{eq:art_chainrule}) in \cite{CP2003} (also see \cite{CP2009}).
For the isotropic case, the chain rule is not needed ({\it cf.} \cite{Car1999,CD2001}).
\end{rem}

\begin{rem}
For a stochastic kinetic solution $u$, we observe that, for any $B^c_R \subset\mathbb{R}$ (the complement of an interval of radius $R$)
and $T\in (0, \infty)$,
\begin{align}\label{eq:kmeas_decay}
\lim_{R \to\infty} \Ex\big[(m^u+n^u)(B^c_R\times \T^d \times [0,T])\big] = 0.
\end{align}
\end{rem}

\begin{rem}
Denote $\bar{\nabla}:= (D_\x, -\pd_\xi)$. Then the two integrals involving the flux function $\FF$ in (\ref{eq:definsol}) can be expressed as
\begin{align*}
\int_0^T \iint H(\xi - u)\,(\FF_u, D_\x\cdot \FF)\cdot \bar{\nabla} \varphi \;\d\xi\d\x\d t,
\end{align*}
which shows clearly the divergence structure attained in this formulation for stochastic kinetic solutions,
so that the integral above can be seen as
\begin{align*}
-\int_0^T \iint \bar{H}(\xi - u)\, (\FF_u, D_\x\cdot\FF)\cdot \bar{\nabla} \varphi \;\d\xi\d\x\d t
\end{align*}
for $\bar{H} := 1 - H$.
\end{rem}

\section{Framework for Continuous Dependence Estimates} \label{sec:I_preparation}

In this section, we develop a general framework for the continuous dependence estimates
of stochastic kinetic solutions.
Consider the pair of nonlinear equations:
\begin{align}
&\pd_t u - \nabla  \cdot (\A(u) \nabla u) + \nabla \cdot \FF(u,\x) =  \sigma(u) \dot{W}, \label{eq:equation1}\\
&\pd_t v - \nabla  \cdot (\B(v) \nabla  v) + \nabla \cdot \GG(v,\x) =  \tau(v) \dot{W},\label{eq:equation2}
\end{align}
where $\B$ is also a positive semi-definite matrix with square root $\be=(\beta_{ij})$.

Corresponding to assumptions (\ref{eq:A_F})--(\ref{eq:A_alpha}) for (\ref{eq:equation1}), we assume the
following conditions for \eqref{eq:equation2} for $\x,\y\in \T^d$:
\begin{align}
&D_\x\cdot\GG_u(\cdot,\x) \in  L^\infty, \label{eq:AB_F}\\[1mm]
&|\GG_u(u,\x) - \GG_u(v,\x)|\leq
C\big(|u|^{p - 1} + |v|^{p -1}+1\big) |u - v|^{\kappa_{G1}},\label{eq:AB_kappa1}\\[1mm]
&|D_\x\cdot \GG(u,\x)  - D_\y\cdot \GG(u, \y)| \leq C\big(|u|^{q }+1\big)|\x - \y|^{\kappa_{G2}},\label{eq:AB_kappa2}\\[1mm]
&|\tau(u) - \tau(v)| \leq C|u - v|^{\lambda_\tau}, \label{eq:AB_sigma}\\[1mm]
&\sup_{i,j}|\beta_{ij}(u) - \beta_{ij}(v)| \leq C|u -v|^{\gamma_\be}.  \label{eq:AB_alpha}
\end{align}
We allow $\kappa_{G1}$, $\kappa_{G2}$, $\lambda_\tau$, and $\gamma_\be$ to be different
from $\kappa_{F1}$, $\kappa_{F2}$, $\lambda_\sigma$, and $\gamma_\ab$, respectively,
but we still assume that $D_\x\cdot \GG(u,\x)$ and $\tau(u)$ have at most linear growth in $u$
and $\B(u)$ has polynomial growth in $u$.
As before, we require that $\kappa_{G1} > 0$, $\kappa_{G2}>0$,
$\lambda_\tau >\frac{1}{2}$, and $\gamma_\be > \frac{1}{2}$.

We employ the Kruzhkov doubling-of-variable technique and attempt to bound the difference
of their stochastic kinetic solutions, so that the kinetic solution $u$
of \eqref{eq:equation1}  is understood to take the spatial variable $\x$, and
the kinetic solution $v$ of \eqref{eq:equation2}
is understood to take the spatial variable $\y$.

In the following, we always assume
$$
u_0, v_0 \in L^p(\Omega, \mathcal{P}, \d\mathbb{P}; L^p(\T^d)) \cap  L^p(\Omega, \mathcal{P}, \d\mathbb{P}; N^{\kappa,1}(\T^d)).
$$

The role of the kinetic function is based on the observation:
\begin{align*}
\int_\mathbb{R} H(\xi - u(\x,t)) (1 - H(\xi - v(\y,t))) \;\d \xi = (v(\y,t) - u(\x,t))_+.
\end{align*}
The manipulations are formally only as they stand directly.
Thus, we have to make mollifications for justification.

\smallskip
Let $\eta_1: \mathbb{R} \to \mathbb{R}$ be defined as a smooth convex function,
equal to $(\cdot)_+$ outside $[-1,1] \subseteq \mathbb{R}$,
and symmetric with respect to the origin in the sense that $\eta'_1(-r) = 1 - \eta'_1(r)$.
Such a function $\eta_1$ can be constructed so that
$\eta'_1(r) := \int_{-\infty}^r \tilde{J}_1(s)\;\d s$,
where $\tilde{J}_1$ is a standard symmetric bump function supported on $[-1,1]$
such as
$\tilde{J}_1(r) = C\exp(\frac{1}{1 - r^2})$
with choice of $C$ as the normalization constant
so that $\int_{\mathbb{R}} \tilde{J}_1(r)\;\d r = 1$.
Now scaling by $\rho$ in the usual way to obtain an approximation to $\delta(r)$:
$\eta''_\rho(r)= \frac{1}{\rho}\eta_1''(\frac{r}{\rho})$,
so that $\eta_\rho'(r)$ preserves the symmetry:
\begin{equation}\label{3.8a}
1 - \eta_\rho'(r) = \eta_\rho'(-r).
\end{equation}
Finally, we set
\begin{align}\label{eq:eta_construction}
\eta_\rho(r) := \int_{-\infty}^r \eta'_\rho(s)\;\d s.
\end{align}
By the symmetry,
we see that $\eta_\rho$ coincides with $(\cdot)_+$ outside $[-\rho,\rho]$.

\smallskip
Using the definition of the Heaviside function $H$ and writing $\bar{H}(\zeta - v):=1 - H(\zeta - v)$,
we have
\begin{align}
&\int_\mathbb{R} \int_{\mathbb{R}} H(\xi - u(\x,t)) \bar{H}(\zeta - v(\y,t))
  \,\eta''_\rho(\zeta-\xi)\;\d\zeta\d\xi\nonumber\\
&=\int_u^\infty \eta'_\rho(v(\y,t) - \xi) \;\d \xi
 =  \eta_\rho(v(\y,t) - u(\x,t)),\label{eq:kinetic_combo}
\end{align}
while $\eta_\rho(v - u)$ approximates $(v(\y,t) - u(\x,t))_+$ as $\rho \to 0$.

\smallskip
Define
\begin{equation}\label{3.11}
J_\theta(\cdot):=\frac{1}{\theta^d}J(\frac{\cdot}{\theta}),
\end{equation}
where $J: \T^d \to \R$ is a smooth symmetric Friedrichs mollifier on $\T^d$.
Then we can further multiply \eqref{eq:kinetic_combo} by $J_\theta(\y  - \x)$
and integrate in $\d \y$ to approximate $(v(\x,t) - u(\x, t))_+$ as $\theta \to 0$.

\smallskip
Before proceeding to the manipulations that mould the equation into a form
similar to the terms in (\ref{eq:kinetic_combo}) above,
we state a lemma that provides a way to leverage definition (\ref{eq:definsol}) into a more versatile form.
This is essentially Proposition 10 of Debussche--Vovelle \cite{DV2010} (see also \cite[Proposition~2.10]{DV2018},
\cite[Proposition 3.1]{DHV2016}, and \cite[Proposition 3.1]{Hof2013}).

\begin{lma}\label{thm:tech_lem}
Let $u$ be a kinetic solution of \emph{(\ref{eq:eq1_parab})}.
Then there exist representatives $f^\pm(\xi,\x; t)$ of $H(\xi - u(\x,t)) = \mathds{1}_{\xi > u(\x,t)}(\xi)$ that are almost
surely left- and right-continuous-in-time.
That is, for any $\psi \in C^2_c(\mathbb{R}\times \T^d)$,
\begin{enumerate}
\item[\rm (i)] for every $\tau \in (0,T]$,
\[
\iint H(\xi - u(\x,\tau \pm \ep)) \psi(\xi,\x)\;\d\x\d\xi
\overset{\ep \to 0}{\longrightarrow} \iint f^{\pm}(\xi,\x,\tau)\psi(\xi,\x)\;\d\x\d\xi
\quad a.s.{\rm ;}
\]
\item[\rm (ii)] for $\tau=0$,
\[
\iint  H(\xi - u(\x, \ep)) \psi(\xi,\x)\;\d\x\d\xi
\overset{\ep \to 0}{\longrightarrow} \iint f^{+}(\xi,\x; 0)\psi(\xi,\x)\;\d\x\d\xi
\quad a.s..
\]
\end{enumerate}
Moreover, for any $t \in [0,T]$ and $\psi \in C^1_{\rm c}(\R\times \T^d)$,
\begin{equation}\label{3.12a}
\iint (f^+ - f^- )(\xi,\x;t) \psi(\xi,\x)\,\d \x \,\d \xi
= - \int_0^T \iint \mathds{1}_{\{t\}}(s)\pd_\xi\psi(\xi,\x)\, \d m^u (\xi, \x, s),
\end{equation}
so that
$f^+ = f^-=H(\cdot - u)$
except on at most a countable subset of $[0,T]$.
\end{lma}

To arrive at \eqref{3.12a}, we have used the following property of the parabolic measure $n^u(\xi,\x,t)$:
For any $t \in [0,T]$ and $\phi\in C_c(\R\times \T^d )$,
\begin{equation}
\begin{aligned}\label{3.12b}
&\int_0^T \iint \mathds{1}_{\{t\}}(s)\phi(\xi,\x)\, \d n^u (\xi, \x, s)\\
&=\int_0^T \iint \mathds{1}_{\{t\}}(s)\phi(\xi,\x)
   \Big|\nabla_\x\cdot \int_0^{u(\x,s)} \ab(\zeta) \;\d \zeta\Big|^2 \;\d\x\d\y\d s
   =0,
\end{aligned}
\end{equation}
since $\nabla_\x\cdot \int_0^{u(\x,s)} \ab(\zeta) \;\d \zeta \in  L^2(\Omega \times \T^d\times [0,T])$.

\medskip
Using the definition of stochastic kinetic solutions in (\ref{eq:definsol}),
we can manipulate to obtain the bounds
for the terms in (\ref{eq:kinetic_combo}) above in the following way:

\smallskip
We first derive a version of (\ref{eq:definsol}) without the temporal integral
by choosing a test function of
form: $\varphi(\xi,\x,s) = \phi(\xi,\x) \chi^\ep(s)$ with
\begin{align}\label{eq:cutoff_time}
\chi^\ep (s)
: = \begin{cases}
1 &\quad \mbox{for $s \leq t$}, \\
1 - \frac{s-t}{\ep} &\quad \mbox{for $t \leq s \leq t + \ep$},\\
0 &\quad \mbox{for $s \geq t + \ep$},
\end{cases}
\end{align}
so that $-\pd_s \chi^\ep$ approximates $\delta_t(s)$  as $\ep \to 0$.

Then, from (\ref{eq:definsol}),
\begin{align}
&\int_0^T \iint  H(\xi - u) \pd_s \left( \phi\,  \chi^\ep\right)\;\d\xi\,\d\x\,\d s
  +\int_0^T \iint H(\xi - u)\,   \FF_u(\xi,\x)\cdot \nabla \phi\,  \chi^\ep \;\d\xi\,\d\x\,\d s\nonumber\\
 &\quad - \int_0^T \iint H(\xi - u)   \, D_\x\cdot\FF(\xi,\x)\, \phi_\xi\, \chi^\ep \;\d\xi\,\d\x\,\d s\nonumber\\
&=-\int_0^T \iint H(\xi - u)   \A(\xi): \nabla^2\phi\, \chi^\ep \;\d \xi\,\d\x\,\d s
   - \int_0^T \iint  \phi_\xi\, \chi^\ep\, \d (m^u + n^u)(\xi,\x, s)\nonumber\\
&\quad +  \frac{1}{2} \int_0^T \int  \sigma^2(u)\phi_u(u,\x)\,\chi^\ep(s)\;\d\x\,\d s
  + \int_0^T \int  \sigma(u)\phi(u,\x)\, \chi^\ep(s) \;\d \x\d W(s)\nonumber \\
&\quad + \iint  H(\xi - u_0)\,\phi(\xi,\x)\;\d\xi\,\d \x.\label{3.11-a}
\end{align}
Taking limit $\ep\to 0$ in both sides of \eqref{3.11-a}, we apply Lemma \ref{thm:tech_lem} to obtain
\begin{align}\label{eq:definsol2}
H^+_u(\phi):=\iint f^+(\xi,\x,t) \phi\;\d\x\d \xi
= I^{u}_0(\phi) +  I^{u}_1(\phi) + B^{u}(\phi),
\end{align}
where $f(\xi,\x, t)$ agrees with $H(\xi - u(\x,t))$ almost surely, except possibly on a countable subset of $[0,T]$, and
\begin{align}
I^u_0(\phi) & =  \iint f^+(\xi,\x,0) \phi \;\d\x\,\d\xi, \label{3.13aa}\\
I^u_1(\phi) &= \int_0^t \iint H(\xi - u)\,\FF_u(\xi,\x)\cdot\nabla_\x\, \phi \;\d\x\,\d\xi\,\d s \notag\\
&\quad   - \int_0^t \iint D_\x\cdot\FF(\xi,\x) H(\xi - u)\, \phi_\xi\;\d\x\,\d\xi\,\d s \notag\\
&\quad + \int_0^t \iint  H(\xi - u)\,  \A (\xi):\nabla^2_\x\phi \;\d\x\,\d\xi\,\d s  \notag\\
&\quad  - \frac{1}{2}\int_0^t \int \sigma^2(u(\x,s)) \phi_u(u(\x,s),\x) \;\d \x\,\d s\notag\\
&\quad + \int_0^{t}\iint \phi_\xi(\xi,\x) \; \d n^u(\xi,\x,s)
  + \int_0^{t+0}\iint \phi_\xi(\xi,\x) \; \d m^u(\xi,\x,s),\nonumber\\
            \label{3.13ba}\\
B^u(\phi) &=  - \int_0^t \int \sigma(u(\x,s))\,\phi(u(\x,s),\x) \;\d\x\,\d W(s). \label{3.13ca}
\end{align}

If $\chi^\ep(s)$ in \eqref{eq:cutoff_time} is replaced by
\begin{align*}
\chi_\ep (s)
: = \begin{cases}
1 &\quad \mbox{for $s \leq t-\ep$}, \\
1 - \frac{s-(t-\ep)}{\ep} &\quad \mbox{for $t-\ep \leq s \leq t$},\\
0 &\quad \mbox{for $s \geq t$},
\end{cases}
\end{align*}
we obtain a similar identity to \eqref{eq:definsol2} with $f^+$ replaced by $f^-$ and the last term in \eqref{3.13ba}
replaced by $\int_0^{t-0}\iint \phi_\xi(\xi,\x) \; \d m^u(\xi,\x,s)$.

\smallskip
As for the analogous equation for $\bar{H}(\zeta - v)=1-H(\zeta- v)$, with analogous representative $\bar{g}^+(\zeta,\y,t)$,
 making the requisite changes in (\ref{eq:eq2_kinetic}) directly,
we obtain
\begin{align}
\bar{H}^+_v(\tilde{\phi}):= \iint \bar{g}^+(\zeta,\y,t)\tilde{\phi}\; \d \y\,\d \zeta
= \bar{I}^v_0(\tilde{\phi}) +  \bar{I}^v_1(\tilde{\phi}) +  \bar{B}^v(\tilde{\phi}),\label{eq:definsol2a}
\end{align}
where
\begin{align}
\bar{I}^v_0(\tilde{\phi}) & = - \iint \bar{g}^+(\zeta,\y,0)\tilde{\phi}\;\d\y\,\d\zeta,  \label{3.14a}\\
\bar{I}^v_1(\tilde{\phi}) &= \int_0^t \iint  \bar{H}(\zeta - v)\,\GG_u(\zeta,\y)\cdot \nabla_\y \tilde{\phi} \;\d\y\,\d\zeta\,\d s\notag\\
&\quad
 - \int_0^t \iint D_\x\cdot\GG(\zeta,\y) \bar{H}(\zeta - v)\, \tilde{\phi}_\zeta \;\d\y\,\d\zeta\,\d s\notag\\
&\quad + \int_0^t \iint \bar{H}(\zeta - v)\, \B(\zeta): \nabla^2_\y\tilde{\phi} \;\d\y\,\d\zeta\,\d s\nonumber\\
&\quad
+ \frac{1}{2}\int_0^t \int \tau^2(v(\y,t))\,\tilde{\phi}_v(v(\y,t),\y) \;\d\y\,\d s
\notag\\
&\quad -\int_0^{t}\iint \tilde{\phi}_\zeta(\zeta,\y)\; \d n^v(\zeta,\y,s)
  -\int_0^{t+0}\iint \tilde{\phi}_\zeta(\zeta,\y)\; \d m^v(\zeta,\y,s),\notag\\
     \label{3.14b}\\
\bar{B}^v(\tilde{\phi}) & = \int_0^t \int \tau(v(\y,t))\,\tilde{\phi}(v(\y,t),\y) \;\d\y\,\d W(s). \label{3.14c}
\end{align}
Then we can find an expression for the left-hand side of (\ref{eq:kinetic_combo})
via (\ref{eq:definsol2})--\eqref{3.14c}
by choosing the test functions that will be subsequently prescribed.

\subsection{Product Estimate}
We can now use the expression for the left-hand side of (\ref{eq:kinetic_combo}).

\begin{prop}\label{thm:lma_product}
Let $u$ be a kinetic solution of \emph{(\ref{eq:equation1})} with initial data $u_0$,
and let $v$ be a kinetic solution of \emph{(\ref{eq:equation2})} with initial data $v_0$.
Let the nonlinear functions in \emph{(\ref{eq:equation1})}--\emph{(\ref{eq:equation2})}
satisfy \emph{(\ref{eq:A_F})}--\emph{(\ref{eq:A_alpha})}
and \emph{(\ref{eq:AB_F})}--\emph{(\ref{eq:AB_alpha})}.
Then
\begin{align}
\Ex\big[\iint f^+(\xi,\x,t)\bar{g}^+(\xi,\x,t)\;\d \xi\, \d \x \big]
\leq
\Ex\big[I^0\big] + \Ex\big[I^a\big] + \Ex\big[I^F\big] + \Ex\big[I^\sigma\big] + \Ex\big[I^\eta\big],
\label{3.13a}
\end{align}
where
\begin{align}
I^0 = & \int  f^+(\xi,\x,0) \bar{g}^+(\zeta,\y,0) \varphi(\xi,\zeta,\x,\y) \;\d E,\label{3.13b}\\
I^a = & \int_0^t \int  \bar{H}(\zeta - v) H(\xi - u)\big(\A(\xi):\nabla^2_{\x} \varphi + \B(\zeta):\nabla^2_{\y} \varphi\big)\;\d E\, \d s \notag\\
& -  \int_0^{t} \iint\int \varphi(\xi,v,\x,\y) \; \d n^u(\xi,\x,s)\,\d\y
 - \int_0^{t} \iint\int \varphi(u,\zeta,\x,\y) \; \d n^v(\zeta,\y,s)\,\d\x,\label{3.13d}\\
I^F = &  \int_0^t \int  \bar{H}(\zeta - v)  H(\xi - u)\big(\FF_u(\xi,\x)  \cdot \nabla_\x \varphi + \GG_u(\zeta, \y) \cdot \nabla_\y \varphi\big)\;\d E \,\d s\notag\\
& - \int_0^t \int  \bar{H}(\zeta - v)H(\xi - u)\big( D_\y\cdot \GG(\zeta,\y) - D_\x\cdot\FF(\xi,\x)\big) \; \varphi_\zeta\;\d E\,\d s,\label{3.13c}\\
I^\sigma = &\,\frac{1}{2}\int_0^t \iint \big(\tau(v(\y,s)) - \sigma(u(\x,s)) \big)^2\varphi(u,v,\x,\y) \;\d \x\,\d\y\,\d s, \label{3.13e}\\
I^\eta = &\iint f^+(\xi,\x,t)\bar{g}^+(\xi,\x,t)\;\d \xi\, \d \x -\int f^+(\xi,\x,t)\bar{g}^+(\zeta,\y,t) \varphi(\xi,\zeta,\x,\y)\;\d E, \label{3.13f}
\end{align}
with $\varphi(\xi,\zeta,\x,\y) := \eta''_\rho(\zeta-\xi) J_\theta(\y - \x)$
and $\d E:=\d\xi\,\d\zeta\,\d\x\,\d\y$.
\end{prop}

\begin{pf}
We divide the proof into five steps.

\smallskip
1. For simplicity of notation,
we write \eqref{eq:definsol2} and \eqref{eq:definsol2a} as
\begin{align}\label{eq:definsol3}
H^+_u = I^u_0 + I^u_1
+ B^u
\end{align}
and
\begin{align}
\bar{H}^+_v = \bar{I}^v_0 + \bar{I}_1^v + \bar{B}^v,
\label{eq:definsol3a}
\end{align}
by dropping the dependence $\phi$ and $\tilde{\phi}$ in these functionals when no confusion arises.

Multiplying (\ref{eq:definsol3}) by (\ref{eq:definsol3a}), we have
\begin{align}
H^+_u\bar{H}^+_v = &\,\int f^+(\xi,\x,t)\bar{g}^+(\zeta,\y,t)\phi(\xi,\x) \tilde{\phi}(\zeta,\y)\;\d E\nonumber\\
=&\,  \int f^+(\xi,\x,t)\bar{g}^+(\zeta,\y,t) \varphi(\xi,\zeta,\x,\y) \;\d E\nonumber\\
=&\,  I^u_0\bar{I}^v_0 + I^u_1\bar{H}^+_v + \bar{I}^v_1 H^+_u
- I^u_1\bar{I}^v_1 + B^u\bar{B}^v +I_0^u \bar{B}^v+\bar{I}_0^v B^u, \label{eq:definsol3b}
\end{align}
where we have denoted $\varphi(\xi,\zeta,\x,\y)=\phi(\xi,\x) \tilde{\phi}(\zeta,\y)$.

\medskip
2. Since $I^u_1$ and $\bar{I}^v_1$ are the processes of finite variation ({\it cf}. \cite[Proposition~0.4.5]{RY2005}),
then integrating by parts yields
\begin{align*}
I^u_1\bar{H}^+_v
&=  \int_0^t \bar{H}^+_v(s)\;\d I^u_1(s)
   + \int_0^t I^u_1(s)\;\d \bar{H}^+_v(s)\notag \\
&=  \int_0^t \bar{H}^+_v(s-)\;\d I^u_1(s)
   + \int_0^t I^u_1(s-)\;\d \bar{H}^+_v(s) + \sum \Delta I^u_1(s) \Delta \bar{H}^+_v(s)\\
&=  \int_0^t \bar{H}^+_v(s-)\;\d I^u_1(s)
   + \left(\int_0^t I^u_1(s-)\; \d \bar{I}^v_1(s) + \int_0^t I^u_1(s-)\; \d \bar{B}^v(s) \right)  \\&\quad\,\,
   + \sum \Delta I^u_1(s) \Delta \bar{H}^+_v(s),\notag
\end{align*}
where the sum is over the countable number of points at which the jumps
are non-zero.
Similarly, we have
\begin{align*}
\bar{I}^v_1H^+_u
& =  \int_0^t H^+_u(s)\;\d \bar{I}^v_1(s)
   + \int_0^t \bar{I}^v_1(s)\;\d H^+_u(s) \notag\\
&= \int_0^t H^+_u (s-)\;\d \bar{I}^v_1(s)
 + \left(\int_0^t \bar{I}^v_1(s-)\; \d I^u_1(s)+ \int_0^t \bar{I}^v_1(s-)\; \d B^u(s) \right)  \\&\quad\,\,
  + \sum \Delta \bar{I}^v_1(s) \Delta H^+_u(s).
\end{align*}
Furthermore, we obtain
\begin{align}\label{eq:IuIv_jump}
I^u_1\bar{I}^v_1
= \int_0^t I^u_1(s-) \;\d \bar{I}^v_1(s)+ \int_0^t \bar{I}^v_1(s-) \;\d I^u_1(s)
  + \sum \Delta I^u_1(s) \Delta \bar{I}^v_1(s).
\end{align}

The only jumps that may occur come from the terms, $m^u(\phi_\xi\times[0,s])$ in $I^u_1$
and $-m^v(\tilde{\phi}_\zeta\times[0,s])$ in $\bar{I}^v_1$, so that
\begin{equation}\label{eq:jump_equal}
\begin{aligned}
 \Delta  H^+_u(s)  =  \Delta I^u_1(s) = \Delta m^u(\phi_\xi\times[0,s]),\\
\Delta  \bar{H}^+_v(s)  =  \Delta \bar{I}^v_1(s) = \Delta m^v(\tilde{\phi}_\zeta\times[0,s]).
\end{aligned}
\end{equation}
Therefore, we have
\begin{align*}
&I^u_1\bar{H}^+_v + \bar{I}^v_1H^+_u- I^u_1\bar{I}^v_1 \\
&=  \int_0^t \bar{H}^+_v(s-)\;\d I^u_1(s)
    + \int_0^t H^+_u(s-) \;\d \bar{I}^v_1(s)
   + \int_0^t I^u_1(s-)\; \d \bar{B}^v(s)
     \\&\quad\,\,
  + \int_0^t \bar{I}^v_1(s-)\; \d B^u(s)
  + \sum \Delta \bar{H}^+_v(s) \Delta H^+_u(s).
\end{align*}
Next we claim that
\begin{eqnarray}
&&\int_0^t \bar{H}^+_v(s-)\;\d I^u_1(s) +  \int_0^t H^+_u(s-) \;\d \bar{I}^v_1(s)
  +\sum \Delta \bar{H}^+_v(s) \Delta H^+_u(s)\nonumber \\
&&=\int_0^t H^+_u(s) \;\d \bar{I}^v_1(s) +  \int_0^t \bar{H}^+_v(s)\;\d I^u_1(s). \label{eq:jumps}
\end{eqnarray}
This can be seen by using \eqref{eq:jump_equal} to obtain
\begin{align*}
\int_0^t \bar{H}^+_v(s)\,\d I^u_1(s) -  \int_0^t \bar{H}^+_v (s-)\,\d I^u_1(s)
&= \int_0^t \bar{I}^v_1(s)\,\d I^u_1(s) - \int_0^t \bar{I}^v_1(s-)\,\d I^u_1(s)  ,\\
\int_0^t H^+_u(s)\,\d \bar{I}^v_1(s)  - \int_0^t H^+_u (s-)\,\d \bar{I}^v_1(s)
&= \int_0^t I^u_1(s)\,\d\bar{I}^v_1(s) - \int_0^t I^u_1(s-)\,\d \bar{I}^v_1(s),
\end{align*}
from which the claim follows by \eqref{eq:IuIv_jump}. With \eqref{eq:jumps}, we can conclude that
\begin{equation}\label{eq:grandproduct}
\begin{aligned}
&\int  f^+(\xi,\x,t)\bar{g}^+(\zeta,\y,t)\varphi(\xi,\zeta,\x,\y)
\;\d E \\
&= I^u_0\bar{I}^v_0+ \int_0^t \iint \bar{H}(\zeta - v(\y,s)) \tilde{\phi}_\zeta
\;\d \zeta\d\y\d I^u_1\\&\quad\,\,
+ \int_0^t \iint H(\xi - u(\x,s)) \phi_\xi
 \;\d\xi\d\x\d \bar{I}^v_1
+ B^u \bar{B}^v + \mathscr{M},\,\,
\end{aligned}
\end{equation}
where $\mathscr{M}$ denotes a martingale term, which has expectation zero.

\smallskip
3. Next we have
\begin{align}
I^u_0\bar{I}^v_0  = \int  f^+(\xi ,\x,0) \bar{g}^+(\zeta,\y,0) \varphi(\xi,\zeta,\x,\y)
\;\d E.
\label{3.19}
\end{align}
By the It\^o isometry,
\begin{align}
\Ex\big[B^u \bar{B}^v\big]
&= - \Ex\big[\int_0^t \int  \sigma(u(\x,s))\phi(u(\x,s),\x)\;\d\x\d W(s)\notag\\
&\qquad\qquad \times
   \,\int_0^t \int  \tau(v(\y,s))\tilde{\phi}(v(\y,s),\y) \;\d\y\d W(s)\big]\notag\\
&=-\Ex\big[\int_0^t \iint \sigma(u)  \tau(v) \varphi(u,v,\x,\y)\;\d\x\d\y\d s\big].
\label{3.18c}
\end{align}

\medskip
4. By a density argument via the monotone class
theorem (see {\it e.g.}, \cite[\S 2.3.1]{Vla2002} or \cite[(24)--(25)]{DV2010} in the context),
we can choose
\begin{equation}\label{3.20-a}
\vp(\xi,\zeta,\x,\y):= \eta''_\rho(\zeta-\xi)J_\theta(\y-\x)\ge 0,
\end{equation}
where  $\eta_\rho$ and $J_\theta$ are defined as in  \eqref{eq:eta_construction} and \eqref{3.11}.
With such a choice of the test function, we have the following usual identities:
\begin{align}\label{eq:testf_identities}
\nabla_\x \varphi +  \nabla_\y \varphi = 0,\quad\,
\nabla ^2_{\x} \varphi -  \nabla^2_{\y} \varphi =0,\quad\,
\varphi_\xi + \varphi_\zeta = 0.
\end{align}

Combining \eqref{3.19}--\eqref{3.18c} with (\ref{eq:grandproduct})
and using $\varphi_\xi = -\varphi_\zeta$,
 we have
\begin{align*}
&\Ex\big[\int f^+(\xi,\x,t)\bar{g}^+(\zeta,\y,t) \varphi(\xi, \zeta,\x, \y)\;\d E\big]\\
&=
\Ex\big[I^0\big]  +\Ex\big[I^a\big] +\Ex\big[I^F\big] +\Ex\big[I^\sigma\big]\\
&\quad\,\, - \int_0^{t+0} \iint\int \varphi(\xi,v^+,\x,\y) \;\d\y\, \d m^u(\xi,\x,s)\\
&\quad\,\,  - \int_0^{t+0} \iint\int \varphi(u^+,\zeta,\x,\y) \; \d\x\,\d m^v(\zeta,\y, s)\\
&\le
\Ex\big[I^0\big]  + \Ex\big[I^a\big] + \Ex\big[I^F\big] + \Ex\big[I^\sigma\big],
\end{align*}
since $m^u$ and $m^v$ are non-negative Radon measures, and $\Ex\big[\mathscr{M}\big]=0$,
where $u^+$ and $v^+$ are the right continuous versions of $u$ and $v$, respectively,
and there is no distinction between $u$ and $u^+$ within a time integral
against a non-atomic measure.

\smallskip
5. By the definition of $I^\eta$,
we conclude \eqref{3.13a}.
\end{pf}

\subsection{Difference Estimates}
From \eqref{3.13a},
we need to estimate the integral terms $\Ex\big[I^0\big]$,
$\Ex\big[I^a\big]$, $\Ex\big[I^F\big]$, $\Ex\big[I^\sigma\big]$, and $\Ex\big[I^\eta\big]$
as defined in \eqref{3.13b}--\eqref{3.13f}.
We refer to these integral terms as the \emph{initial term},  \emph{parabolic term}, \emph{flux term},
 \emph{It\^o correction term}, and \emph{mollification term}, respectively.

\begin{prop}\label{thm:flux_parabolic_bound}
Let $u$ be a kinetic solution of \emph{(\ref{eq:equation1})} with initial data $u_0$,
and let $v$ be a kinetic solution of \emph{(\ref{eq:equation2})} with initial data $v_0$.
Let $\eta_\rho$ and  $J_\theta$ be defined as in \eqref{eq:eta_construction} and \eqref{3.11}.
Let $\ab$ and $\be$ satisfy \eqref{eq:A_alpha} with respective indices $\gamma_\ab$ and $\gamma_\be$,
and let $\sigma$ and $\tau$ satisfy \eqref{eq:A_sigma} with respective indices $\lambda_\sigma$ and $\lambda_\tau$.
Assume that
\begin{align}
&\|\sqrt{\B} - \sqrt{\A} \|_{L^\infty} : =\sup_{i,j}\| \beta_{ij} - \alpha_{ij}\|_{L^\infty}<\infty, \label{eq:A_alpha_beta}\\
&\|(\GG_u - \FF_u ,  D_\x\cdot (\GG - \FF),  \tau -  \sigma )\|_{L^\infty} < \infty. \label{eq:A_F_G}
\end{align}
Then the following estimates hold\,{\rm :}
\begin{itemize}
\item[\rm (i)] For the parabolic term,
\begin{align}
\Ex\big[I^a\big]
\leq &\,
d \theta^{-2}\big(\| \sqrt{\B} - \sqrt{\A} \|^2_{L^\infty}
  + C(\ab)\rho^{\gamma_\ab}(\| \sqrt{\B} - \sqrt{\A}\|_{L^\infty}
   +\rho^{\gamma_\ab})\big)\nonumber\\
& \,\, \times \Ex\big[\int_0^t \iint  \eta_\rho(v(\y,s) - u(\x,s)) J_\theta(\y-\x)\;\d\x\d\y\d s\big],
\label{eq:parabolic_cde}
\end{align}
where $C(\ab) \ge \|\ab\|_{C^{\gamma_\ab}}$.

\smallskip
\item[\rm (ii)] For the flux term,
\begin{align}
\Ex\big[I^F\big]
\leq&\,
 C\|D_\x \cdot \FF_u\|_{L^\infty} \Ex\big[\int_0^t\iint \eta_\rho(v(\y,s)-u(\x,s)) J_\theta(\y-\x) \;\d\x\d\y\d s\big]\notag\\
&\, + C\big(\theta^{-1}\|\GG_u - \FF_u \|_{L^\infty}
 + \rho^{-1}\|D_\x\cdot (\GG - \FF)\|_{L^\infty}\big)\notag\\
&\quad\,\,\,\,\,\times \Ex\big[\int_0^t \iint  \eta_\rho(v(\y,s) - u(\x,s)) J_\theta(\y-\x)\;\d\x\d\y\d s\big]\notag\\
&+ C\big(\rho^{\kappa_{F1}}\theta^{-1} + \theta^{\kappa_{F2}}\big)\Ex\big[ \int_0^t\int \big(|(u,v)|^p + |(u,v)|^q +1\big) \;\d\x\d s\big],
\label{eq:flux_cde}
\end{align}
where $C$ depends on $d$, $|\T^d|$, $\FF$, and $\GG$.

\smallskip
\item[\rm (iii)] For the It\^o correction term,
\begin{align}\label{eq:correction_cde}
\Ex[I^\sigma] \leq
 C t \rho^{-1} \big(\|\tau - \sigma \|_{L^\infty}^2 +  \rho^{2\lambda_\sigma}\big),
\end{align}
where $C$ is a constant depending on $d$, $|\T^d|$, $\sigma$, and $\tau$.

\smallskip
\item[\rm (iv)] For the mollification term,
\begin{align}
\Ex[I^\eta]& = o_{\theta,\rho}(1) \to 0 \qquad \mbox{as $\theta, \rho\to 0$}.
\label{eq:BV_necessity}
\end{align}
\end{itemize}
\end{prop}

\begin{pf} We divide the proof into four steps.

\smallskip
1. {\it Parabolic terms}.
With reference to (\ref{3.13d}) where $I^a$ is defined, we first show
\begin{align}
I^a
\leq & \int_0^t \int  \bar{H}(\zeta - v) H(\xi - u) \bk{\be(\zeta) -\ab(\zeta) }\bk{\be(\zeta) - \ab(\zeta) } : \nabla^2_{\x\y} \varphi\;\d E\d s\notag\\
& +\int_0^t \int  \bar{H}(\zeta - v) H(\xi - u) \bk{\be(\zeta) - \ab(\zeta)}\bk{\ab(\zeta) - \ab(\xi)}: \nabla^2_{\x\y} \varphi\;\d E\d s\notag\\
&+\int_0^t \int  \bar{H}(\zeta - v) H(\xi - u)   \bk{\ab(\zeta) - \ab(\xi)} \bk{\be(\zeta) - \ab(\zeta) } : \nabla^2_{\x\y} \varphi\;\d E\d s\notag\\
& +\int_0^t \int  \bar{H}(\zeta - v) H(\xi - u)  \bk{\ab(\zeta) - \ab(\xi)}\bk{\ab(\zeta) - \ab(\xi)} : \nabla^2_{\x\y} \varphi\;\d E\d s.
\label{eq:parabolic_rearrangement}
\end{align}
First, by \eqref{eq:testf_identities}, we have
\begin{align}
&\int_0^t \int  \bar{H}(\zeta - v) H(\xi - u)\big(\A(\xi): \nabla^2_{\x} \varphi +  \B(\zeta): \nabla^2_{\y}\varphi\big)\;\d E\d s\notag\\
&=  -\int_0^t \int  \bar{H}(\zeta - v) H(\xi - u)\big(\A(\xi)- \ab(\xi) \be(\zeta) - \be(\zeta)\ab(\xi) +  \B(\zeta)\big):\nabla_{\x\y}^2 \varphi\;\d E\d s\notag\\
&\,\quad -\int_0^t \int  \bar{H}(\zeta - v) H(\xi - u)\big(\ab(\xi) \be(\zeta) + \be(\zeta)\ab(\xi)\big):\nabla_{\x\y}^2 \varphi\;\d E\d s.\label{3.32a}
\end{align}

Using the chain rule (\ref{eq:art_chainrule})
for kinetic solutions and the symmetry of $\nabla_{\x\y}^2 \varphi$, we have
\begin{align}
&\int_0^t \int \bar{H}(\zeta - v) H(\xi - u) \big(\ab(\xi)\be(\zeta) + \be(\zeta)\ab(\xi) \big) : \nabla^2_{\x\y} \varphi\;\d E\d s\notag\\
&=  \int_0^t \int  \nabla_\x H(\xi - u)\otimes\nabla_\y \bar{H}(\zeta - v): \ab(\xi) \be(\zeta)\,\varphi \;\d E \d s\notag\\
&\quad +\int_0^t \int \nabla_\y \bar{H}(\zeta - v) \otimes \nabla_\x H(\xi - u) :\ab(\xi)\be(\zeta)\,\varphi \;\d E\d s\notag\\
&= -\int_0^t\iint \nabla^2_{\y\x}:\int_{\infty}^u \ab(\xi) \int_{-\infty}^v \be(\zeta) \eta''_\rho(\zeta-\xi)
  \;\d\zeta\,\d  \xi\; J_\theta(\y-\x)\;\d\x\d\y\d s\notag\\
 &\quad - \int_0^t \iint\nabla_{\y\x}^2 : \int_{\infty}^u \int_{-\infty}^v \ab(\xi)\be(\zeta) \eta''_\rho(\zeta-\xi)
   \;\d \zeta\,\d \xi\; J_\theta(\y-\x) \;\d\x\d\y\d s\notag\\
&= -2\int_0^t \iint \nabla_\x\otimes\nabla_\y :\Big(\int_0^u \ab(\xi)\;\d \xi\Big)\Big(\int_0^v \be(\zeta)\;\d \zeta\Big)\varphi(u,v,\x,\y)\;\d\x\d\y\d s,
 \label{eq:aebe_equation}
\end{align}
where we have also used the following fact:
$$
\nabla_\x\cdot\big(\int_r^u \ab(\xi) \eta''_\rho(\zeta-\xi) \;\d \xi\big)
=  \nabla_\x \cdot \big(\int_0^u \ab(\xi) \eta''_\rho(\zeta-\xi) \;\d \xi\big)
\qquad\mbox{for any fixed $r$}.
$$

Next, we employ form (\ref{eq:pdefectmeas}) of the parabolic defect measure with $\varphi \in C^\infty_c(\R^2\times (\T^d)^2)$
to obtain
\begin{align*}
&\int_0^{t} \iint\varphi(\xi,v(\y,s),\x,\y)\;\d \y \,\d n^u(\xi,\x,s) \\
&=\int_0^t  \iint \eta''_\rho(v(\y,s)-u(\x,s)) J_\theta(\y-\x) \Big|\nabla_\x\cdot \int_0^{u(\x,s)} \ab(\zeta) \;\d \zeta\Big|^2 \;\d\x\d\y\d s,
\end{align*}
where we have used that $\nabla_\x\cdot \int_0^{u(\x,s)} \ab(\zeta) \;\d \zeta\in L^2(\Omega \times \T^d\times [0,T])$.

Similarly, we have
\begin{align*}
&\int_0^{t} \iint\int \varphi(u(\x,s),\zeta,\x,\y)\,\d\x \d n^v(\zeta, \y, s)\\
&= \int_0^t  \iint \eta''_\rho(v(\y,s)-u(\x,s)) J_\theta(\y-\x)  \Big|\nabla_\y\cdot \int_0^{v(\y,s)} \mathbf{\be}(\xi) \;\d\xi\Big|^2 \;\d\x\d\y\d s,
\end{align*}
by using $\nabla_\y\cdot \int_0^{v(\y,s)} \mathbf{\be}(\xi) \;\d\xi\in L^2(\Omega \times \T^d\times [0,T])$.

\smallskip
Therefore, we obtain
\begin{align*}
&-\int_0^t \int  \bar{H}(\zeta - v) H(\xi - u)\big(\ab(\xi) \be(\zeta) +\be(\zeta) \ab(\xi)\big): \nabla^2_{\x\y} \varphi\;\d E \d s \\
&\quad  -  \int_0^t \iint \int \varphi(\xi,v,\x,\y) \;\d\y \d n^u(\xi, \x, s)
  - \int_0^t \iint \int \varphi(u,\zeta,\x,\y) \;\d\x\d n^v(\zeta,\y, s)\\
&= 2\int_0^t \nabla_{\x\y}^2 :\Big(\int_0^u \ab(\xi)\;\d \xi\Big) \Big(\int_0^v \be(\zeta)\;\d \zeta\Big)\varphi(u,v,\x,\y)\;\d\x\d\y\d s\\
&\quad - \int_0^t  \iint  \Big(\Big|\nabla_\x\cdot \int_0^u \ab(\zeta) \;\d \zeta\Big|^2
+\Big|\nabla_\y\cdot \int_0^v \be(\xi) \;\d \xi\Big|^2\Big)\varphi(u,v,\x,\y) \;\d\x\d\y\d s
\leq 0.
\end{align*}

Inserting this into \eqref{3.32a} yields
\begin{align}
I^a \leq &\, \int_0^t \int  \bar{H}(\zeta - v) H(\xi - u)\big(\A(\xi) - \ab(\xi) \be(\zeta) - \be(\zeta) \ab(\xi) + \B(\zeta)\big): \nabla^2_{\x\y}\varphi\;\d E\d s\notag \\
=&\, \int_0^t \int  \bar{H}(\zeta - v) H(\xi - u)\big(\ab(\xi)  - \be(\zeta)\big)\big( \ab(\xi) -\be(\zeta)\big):\nabla^2_{\x\y}\varphi\;\d E\d s.
\label{eq:aebeequation2}
\end{align}

Notice that
\begin{align*}
&\big(\be(\zeta)-\ab(\xi)\big)\big(\be(\zeta)-\ab(\xi)\big)\\
&=\bk{\be(\zeta)-\ab(\zeta)}\bk{\be(\zeta)-\ab(\zeta)} + \bk{\be(\zeta)-\ab(\zeta)}\bk{\ab(\zeta) - \ab(\xi)}\\
&\quad\, +\bk{\ab(\zeta) - \ab(\xi)} \bk{\be(\zeta)-\ab(\zeta)} + \bk{\ab(\xi)-\ab(\zeta)}\bk{\ab(\xi) - \ab(\zeta)}.
\end{align*}
Combining this with \eqref{eq:aebeequation2}, we complete the proof of (\ref{eq:parabolic_rearrangement}).

\medskip
These terms in $I^a$ can be estimated by invoking either the boundedness of $\big\|\sqrt{\B}-\sqrt{\A}\big\|_{L^\infty}$
or the continuity of $\ab$ and $\be$ in (\ref{eq:A_alpha}) as follows:
\begin{align}
& \left|\int_0^t \int  \bar{H}(\zeta - v) H(\xi - u) \bk{\be(\zeta) -\ab(\zeta) }\bk{\be(\zeta) - \ab(\zeta) }:\nabla^2_{\x}
   \varphi\;\d E\d s\right|\notag\\
&\leq C \int_0^t \int  \bar{H}(\zeta - v) H(\xi - u) \big\|\sqrt{\B} - \sqrt{\A}\big\|^2_{L^\infty} \theta^{-2}
   J_\theta(\y-\x)\eta''_\rho(\zeta - \xi)\;\d E\d s\notag\\
&\leq C \theta^{-2} \big\|\sqrt{\B}-\sqrt{\A}\big\|^2_{L^\infty}\int_0^t \iint \eta_\rho(v - u) J_\theta(\x - \y)\;\d\x\d\y\d s.\label{eq:rev3}
\end{align}
Using $\gamma_\ab$ as the H\"older exponent of $\ab$, we have the estimates:
\begin{align}
&\left|\int_0^t \int  \bar{H}(\zeta - v) H(\xi - u)  \bk{\ab(\zeta) - \ab(\xi) }\bk{\be(\zeta) - \ab(\zeta) }:\nabla^2_{\x} \varphi\;\d E\d s\right|\notag\\
&\leq  d\big\|\sqrt{\B}-\sqrt{\A}\big\|_{L^\infty}\notag\\
&\quad\,\,\,\times \int_0^t \int  \bar{H}(\zeta - v) H(\xi - u)
 \frac{| \ab(\zeta) -\ab(\xi)|}{|\zeta - \xi|^{\gamma_\ab}}\rho^{\gamma_\ab} |\nabla^2_{\x} J_\theta| \eta''_\rho( \zeta - \xi)\;\d E \d s\notag\\
&\leq d C(\be)\rho^{\gamma_\ab}\theta^{-2}
  \big\|\sqrt{\B} - \sqrt{\A}\big\|_{L^\infty}\int_0^t \iint \eta_\rho(v - u) J_\theta(\y-\x)\;\d\x\d\y\d s,
\label{eq:rev4}
\end{align}
and similarly,
\begin{align*}
&\left|\int_0^t \int  \bar{H}(\zeta - v) H(\xi - u) \bk{\be(\zeta) - \ab(\zeta)} \bk{\ab(\zeta) - \ab(\xi)}:\nabla^2_{\x} \varphi\;\d E\d s\right|\notag\\
&\leq  d C(\ab)\rho^{\gamma_\ab}\theta^{-2}\big\|\sqrt{\B} - \sqrt{\A}\big\|_{L^\infty}\int_0^t \iint \eta_\rho(v - u) J_\theta(\y-\x)\;\d\x\d\y\d s,
\end{align*}
where $C(\ab) \ge \|\ab\|_{C^{\gamma_\ab}}$, and we have used the estimate:
$|\nabla_\x^2 J_\theta(\x)|
\le C \theta^{-2} J_\theta(\x)$.
An analogous estimate also holds: $|\nabla_\x J_\theta(\x)|\le C \theta^{-1} J_\theta(\x)$ --- which will
be used below.
Finally, we have the estimate:
\begin{align}
&\left|\int_0^t \int  \bar{H}(\zeta - v) H(\xi - u)\bk{\ab(\xi) - \ab(\zeta)}\bk{\ab(\xi) - \ab(\zeta)}:\nabla^2_{\x}\varphi\;\d E\right| \;\d s\notag\\
&\leq d C(\ab) \rho^{2\gamma_\ab}\theta^{-2} \int_0^t \iint  \eta_\rho(v - u) J_\theta(\y - \x)\;\d\x\d\y\d s.
\label{eq:rev5}
\end{align}
Combining all the estimates above to conclude \eqref{eq:parabolic_cde}.

\medskip
2. {\it Flux terms}.
Notice that
\begin{align}
&\FF_u(\xi,\x) \cdot \nabla_\x \varphi + \GG_u(\zeta, \y)\cdot \nabla_\y \varphi\notag\\
& = \big(\GG_u(\zeta,\y)  - \FF_u(\zeta, \y)\big)\cdot \nabla_\y \varphi +\big(\FF_u(\zeta,\y)  - \FF_u(\xi, \y)\big)\cdot \nabla_\y \varphi \notag\\
&\quad\, + \big(\FF_u(\xi,\y)  - \FF_u(\xi, \x)\big)\cdot \nabla_\y \varphi +  \FF_u(\xi, \x)\cdot\big( \nabla_\x \varphi +  \nabla_\y \varphi\big),
   \label{eq:3.13c_expanded-a}  \\[2mm]
& D_\y\cdot\GG(\zeta,\y) - D_\x\cdot\FF(\xi,\x) \notag\\
&= \big( D_\y\cdot\GG(\zeta,\y) - D_\y\cdot\FF(\zeta,\y)\big)+\big(D_\y\cdot\FF(\zeta,\y) - D_\x\cdot\FF(\zeta,\x)\big)\notag\\
&\quad\, +\big(D_\x \cdot\FF(\zeta,\x) - D_\x\cdot\FF(\xi,\x)\big).
 \label{eq:3.13c_expanded-b}
\end{align}

First, with condition \eqref{eq:A_F_G},
we have
\begin{align}
 &\left|\int_0^t \int  \bar{H}(\zeta - v)  H(\xi - u) \big(\GG_u(\zeta,\y)  - \FF_u(\zeta, \y)\big)\cdot \nabla_\y \varphi \;\d E\d s\right|\notag\\
 &\,\,\,\leq C \theta^{-1}\|\GG_u - \FF_u\|_{L^\infty}\int_0^t \iint  \eta_\rho(v - u) J_\theta(\y - \x)\;\d\x\d\y \d s,\label{eq:rev1} \\
&\left| \int_0^t \int  \bar{H}(\zeta - v)H(\xi - u)\big(D_\x\cdot\GG(\zeta,\y) - D_\x\cdot\FF(\zeta,\y)\big) \;\varphi_\zeta\;\d E\d s\notag\right| \notag\\
&\,\,\,\leq C \rho^{-1}\|D_\x\cdot (\GG -\FF)\|_{L^\infty}\int_0^t \iint  \eta_\rho(v - u) J_\theta(\y - \x)\;\d\x\d\y\d s. \label{eq:rev2}
\end{align}

Next, for $\Gamma(\xi, \zeta,\y) = |\FF_u(\zeta, \y) - \FF_u(\xi, \y)||\zeta -\xi |^{-\kappa_{F1}}$  with
$$
\Gamma(\xi, \zeta, \y) \leq C\big( |\xi|^{p - 1} + |\zeta|^{p - 1} + 1\big),
$$
consider the following calculation:
\begin{align*}
&\iint \bar{H}(\zeta - v) H(\xi - u) \Gamma(\xi,\zeta,\y)|\zeta - \xi|^{\kappa_{F1}} \eta_\rho''( \zeta - \xi) \;\d\xi\d\zeta\\
&=\int_u^\infty \int_{-\infty}^v \Gamma(\xi,\zeta,\y)|\zeta - \xi|^{\kappa_{F1}} \eta_\rho''(\zeta - \xi) \;\d\zeta\d\xi.
\end{align*}
From the bounds on $\Gamma$ and $\eta_\rho''$, by changing variable $\zeta' = \zeta - \xi$, we further have
\begin{align}
& \int_u^\infty \int_{-\infty}^v\Gamma(\xi,\zeta,\y)|\zeta - \xi|^{\kappa_{F1}} \eta_\rho''(\zeta - \xi) \;\d \zeta\d\xi\notag\\
&=  \int_{u}^{\infty} \int_{-\infty}^{v-\xi} \Gamma(\xi,\zeta' + \xi,\y)|\zeta'|^{\kappa_{F1}} \eta_\rho''(\zeta') \;\d \zeta'\d\xi \notag\\
&\leq  C\rho^{\kappa_{F1}}\int_{u}^\infty \sup_{\{|\zeta'| < \rho,\, \zeta' <  v- \xi\}}\Gamma(\xi, \zeta' + \xi,\y) \;\d \xi \notag\\
&\leq  C \rho^{\kappa_{F1}} \int_u^{v- \rho}\big(|v|^{p - 1} + |\xi|^{p - 1}+1\big) \;\d \xi \notag\\
&\leq  C \rho^{\kappa_{F1}} \big(|(u, v)|^p+1 \big). \label{eq:DHV_est}
\end{align}
Now applying bound (\ref{eq:DHV_est}) yields
\begin{align}
&  \left|\int_0^t \int  \bar{H}(\zeta - v)  H(\xi - u)\big( \FF_u(\zeta, \y) - \FF_u(\xi,\y)  \big)\cdot \nabla_\y \varphi \;\d E\d s\right|\notag\\
&\leq C \int_0^t \iint \big(|(u,v)|^p+1\big)\rho^{\kappa_{F1}} \theta^{-1} J_\theta(\y - \x) \;\d\x\d\y\d s\notag\\
&\leq  C{\rho^{\kappa_{F1}}\theta^{-1}} \int_0^t \int \big(|(u,v)|^p+1\big)\;\d\x\d s. \label{eq:FF_est1}
\end{align}

For the next integral, we use the fact that $\varphi_\zeta=-\varphi_\xi$:
\begin{align}
& \left|\int  \bar{H}(\zeta - v)H(\xi - u)\big(D_\y\cdot\FF(\zeta,\y) - D_\x\cdot\FF(\zeta,\x)\big) \;\varphi_\zeta\;\d E\right|\notag\\
&=  \left|\int  \bar{H}(\zeta - v)\,\pd_\xi H(\xi - u)\big(D_\y\cdot\FF(\zeta,\y) - D_\x\cdot\FF(\zeta,\x)\big) \;\varphi\;\d E\right|\notag\\
&\leq \theta^{\kappa_{F2}}\iint \int \bar{H}(\zeta - v)  (|\zeta|^q+1)\;\eta''_\rho(\zeta - u)\;\d\zeta \; J_\theta(\y - \x)\;\d\zeta\d\x\d\y\notag\\
&\le   C\theta^{\kappa_{F2}} \int \big(|(u,v)|^q + 1\big)\;\d\x.\label{eq:FF_est2}
\end{align}

Furthermore, we have
\begin{align}
&  \left| \int_0^t \int  \bar{H}(\zeta - v)  H(\xi - u)\big(\FF_u(\xi,\y)  - \FF_u(\xi, \x)\big)\cdot \nabla_\y \varphi \;\d E\d s\right|\notag\\
&\leq C\|D_\x \cdot \FF_u\|_{L^\infty} \int_0^t\int  \bar{H}(\zeta - v)  H(\xi - u)\eta''_\rho(\zeta - \xi) J_\theta(\y - \x) \;\d E\d s\notag\\
&\leq C \|D_\x \cdot \FF_u\|_{L^\infty} \int_0^t\iint \eta_\rho( v(\y,s) -u(\x,s) ) J_\theta(\y - \x) \;\d\x\d\y\d s.
\label{eq:FF_est3}
\end{align}

Following (\ref{eq:FF_est3}) above, we again have
\begin{align}
& \left| \int_0^t \int  \bar{H}(\zeta - v)H(\xi - u)\big(D_\x\cdot\FF(\zeta,\x) - D_\x\cdot\FF(\xi,\x)\big) \;\varphi_\zeta\;\d E\d s\right|\notag\\
&\leq  C\int_0^t\int \|D_\x \cdot \FF_u\|_{L^\infty} \bar{H}(\zeta - v)H(\xi - u)\eta''_\rho(\zeta-\xi) J_\theta(\y-\x)\;\d E\d s\notag\\
&= C \|D_\x \cdot \FF_u\|_{L^\infty} \int_0^t\iint \eta_\rho( v(\y,s) -u(\x,s)) J_\theta( \y - x) \;\d\x\d\y\d s.
\label{eq:FF_est4}
\end{align}

Finally, using $\nabla_\x \varphi + \nabla_\y \varphi  = 0$ and adding (\ref{eq:rev1})--(\ref{eq:FF_est4}) together,
we obtain
\begin{align*}
&\left|\int_0^t \int  \bar{H}(\zeta - v)  H(\xi - u)\big(\FF_u(\xi,\x) \cdot \nabla_\x \varphi + \GG_u(\zeta, \y)\cdot \nabla_\y \varphi\big) \;\d E\d s\right|\\
&  + \left|\int_0^t \int \bar{H}(\zeta - v) H(\xi - u) \bk{D_\x \cdot \FF(\xi,\x) - D_\y \cdot \GG(\zeta , \y) } \varphi_\xi \;\d E\d s\right|\\
&\leq  C\big(\|\GG_u-\FF_u\|_{L^\infty}\theta^{-1}
   +\|D_\x\cdot (\GG-\FF)\|_{L^\infty}\rho^{-1}\big)\\
&\quad\,\,\,\,\times  \int_0^t \iint  \eta_\rho(v - u) J_\theta(\y-\x)\;\d\x\d\y\d s\\
&\quad + C\|D_\x \cdot \FF_u\|_{L^\infty} \int_0^t\iint \eta_\rho(v(\y,s)-u(\x,s))  J_\theta(\y-\x)\;\d\x\d\y\d s\\
&\quad + C\big(\rho^{\kappa_{F1}}\theta^{-1} + \theta^{\kappa_{F2}}\big) \int_0^t\int \big(|(u,v)|^p + |(u,v)|^q +1 \big)\;\d\x\d s.
\end{align*}

\smallskip
3. {\it It\^o correction term}.
The It\^o correction integral can be estimated as follows:
\begin{align*}
\Ex[I^\sigma]  &= \frac{1}{2}\Ex\big[\int_0^t \iint\big(\tau(v(\y,s)) - \sigma(u(\x,s))\big)^2 \varphi(u(\x,s),v(\y,s),\x,\y) \;\d\x\d\y\d s\big]\\
&\leq  \Ex\big[\int_0^t \iint \bk{ (\tau(v) - \sigma(v))^2 + (\sigma (v) - \sigma (u))^2}\eta''_\rho(v - u) J_\theta(\y - \x) \;\d\x\d\y\d s\big]\\
&\leq  \Ex\big[\int_0^t \iint\big( \|\tau  - \sigma\|^2_{L^\infty}
+ \frac{(\sigma(v) - \sigma(u))^2}{|v - u|^{2\lambda_\sigma}} \rho^{2\lambda_\sigma} \big)\eta''_\rho(v - u) J_\theta(\y - \x) \;\d\x\d\y\d s\big]\\
&\leq C \rho^{-1}
  \| \tau - \sigma\|_{L^\infty}^2
     \int_0^t\iint  J_\theta(\y - \x)\;\d\y\d\x\d s\\
  &\quad\,\,+ C_\sigma \rho^{2\lambda_\sigma} \Ex\big[ \int_0^t
   \iint  \eta''_\rho(v - u) J_\theta(\y - \x)\;\d\x\d\y\d s\big]\notag\\
&\leq C t\,\big(\rho^{-1}\|\tau  - \sigma\|_{L^\infty}^2+ \rho^{2\lambda_\sigma - 1}\big)|\mathbb{T}^d|.
\end{align*}
In the above, $\sigma(u)$ and $\tau(v)$ are symmetric.

\smallskip
4. {\it Mollification term.}
We now follow the argument in the proof of \cite[Theorem~3.2]{DV2018} closely to obtain the estimate in this step,
which will be further refined in Corollary~\ref{thm:I_eta_refine}.

First we decompose the difference:
\begin{align*}
I^\eta &= \iint  f^+(\xi,\x,t) \bar{g}^+(\xi,\x,t) \;\d\xi\, \d \x
           -\int f^+(\xi,\x,t) \bar{g}^+(\zeta,\y,t)\varphi(\xi,\zeta,\x,\y)\,\d E\\
& = I_1^\eta(t) + I_2^\eta(t),
\end{align*}
where
\begin{align*}
I_1^\eta(t):=&\, \iint  f^+(\xi,\x,t) \bar{g}^+(\xi,\x,t) \;\d\xi\, \d \x \\
&\,-  \iint \int f^+(\xi,\x,t) \bar{g}^+(\xi,\y,t) J_\theta(\x -\y) \;\d \xi \,\d \x \,\d \y,
\end{align*}
and
\begin{align*}
I_2^\eta(t) :=& \iint \int f^+(\xi,\x,t) \bar{g}^+(\xi,\y,t) J_\theta(\x -\y) \;\d \xi \,\d \x \,\d \y \\
  & -\int f^+(\xi,\x,t) \bar{g}^+(\zeta,\y, t) \varphi(\xi,\zeta,\x,\y)\;\d E.
\end{align*}

We first have
\begin{align*}
 |I_1^\eta|
 & = \Big|\iint  f^+(\xi,\x,t) \Big(\bar{g}^+(\xi,\x,t)-\int  \bar{g}^+(\xi,\y,t) J_\theta(\x -\y) \,\d y\Big) \;\d\xi \d \x\Big|\\
 & \le  \iint\Big|\chi_{\bar{g}^+}(\xi,\x,t)-\int  \chi_{\bar{g}^+}(\xi,\y,t) J_\theta(\x -\y) \,\d y\Big| \;\d\xi \d \x,
\end{align*}
where $\chi_{\bar{g}^+}=\bar{g}^+(\xi,\x,t) - \mathds{1}_{\xi < 0}$ is integrable in both $\xi$ and $\x$, which is the kinetic functions used in, e.g., \cite{CP2003}.
From the Lebesgue differentiation theorem, the absolute value in the last line tends to zero as $\theta \to 0$, except on an $\x$-measure zero set.
Since $\chi_{\bar{g}^+}$ is integrable in both $\xi$ and $\x$, we can take this limit outside the $(\xi,\x)$-integrals.
This implies that
\begin{align}\label{eq:I_eta_1}
\lim_{\theta \to 0 }\abs{I^\eta_1} = 0  \qquad\,\,\mbox{almost surely}.
\end{align}

For any test function $\psi \in C^\infty_0(\R)$, we have
\begin{align*}
&\int \iint  \psi''(\zeta - \xi ) f^+(\xi,\x,t) \bar{g}^+(\zeta,\x,t) \;\d\xi\,\d \zeta\, \d \x\\
&= - \int \iint \psi(\zeta - \xi) \d_\xi f^+(\xi,\x,t) \d_\zeta\bar{g}^+(\zeta,\x,t) \, \d \x  ,
\end{align*}
where
both $\d_\xi f^+(\xi,\x,t)$ and $-\d_\zeta\bar{g}^+(\zeta,\x,t)$ have the unit mass for each $(\x, t)$.

From our construction of $\eta_\rho$ in \eqref{3.8a}--\eqref{eq:eta_construction}, we see that, for $\varpi  > -\rho$,
\begin{align*}
 (\varpi + \rho)_+ - \eta_\rho(\varpi)
& = \int_{-\rho}^\varpi \int_r^\infty \eta''_\rho(s) \,\d s\,\d r
= \int_{-1}^\infty \big((\varpi + \rho)_+ \wedge (s \rho + \rho) \big)\eta''(s) \,\d s.
\end{align*}
Taking $\varpi = \zeta - \xi$, we find that, for any $s > -1$,
\begin{align*}
(\zeta + \rho - \xi)_+\wedge (s \rho + \rho)  &\le (\zeta + \rho)_+\wedge (s \rho + \rho)  + (\xi)_+ \wedge(s \rho + \rho) \\
& \le \big(  (\zeta + \rho)_+ + (\xi)_+\big) \wedge \big(2\rho(s + 1)\big).
\end{align*}

Therefore, we obtain that
\begin{align*}
0 &\le \iint f^+(\xi,\x,t) \bar{g}^+(\xi - \rho,\y, t)\;\d\xi- \iint \eta''_\rho(\zeta - \xi) f^+(\xi,\x,t) \bar{g}^+(\zeta,\y,t) \;\d\xi\,\d \zeta\\
& = - \iint \big((\zeta + \rho - \xi)_+ - \eta_\rho(\zeta - \xi)\big) \d_\xi f^+(\xi,\x,t) \d_\zeta\bar{g}^+(\zeta,\y,t) \\
& \le   \int_{-1}^\infty \big(\big|\int (\xi)_+ \d_\xi f^+(\xi,\x,t)\big|
    + \big|\int (\zeta + \rho)_+\d_\zeta \bar{g}^+(\zeta,\y,t)\big|\big) \wedge\big(2\rho(s + 1)\big)\,\eta''(s) \,\d s\\
&\le 2 \rho\int_{-1}^\infty (s + 1) \eta''(s) \,\d s\\
& \le C \rho,
\end{align*}
because $\eta''(s)$ is supported on $[-1,1]$.
Integrating the foregoing against $J_\theta(\x - y)\,\d\x\,\d\y$ yields that
\begin{align}\label{eq:I_eta_2}
\abs{I^\eta_2} \le C\rho \iint J_\theta(\x - \y)\,\d \x \,\d \y = C\rho,
\end{align}
where constant $C$ is in fact deterministic.

Putting together \eqref{eq:I_eta_1} and \eqref{eq:I_eta_2}, we obtain that
\begin{align*}
&\lim_{\theta, \rho \to 0} \abs{\iint  f^+(\xi,\x,t) \bar{g}^+(\xi,\x,t) \;\d\xi\, \d \x  - \int f^+(\xi,\x,t) \bar{g}^+(\zeta,\y,t)\varphi(\xi,\zeta,\x,\y)\,\d E} \\
&= 0
\end{align*}
almost surely.
From the uniform boundedness of $\norm{J_\theta}_{L^1(\T^d)}$ and $\norm{\eta_\rho}_{L^1(\R)}$,
we can move the limit outside the expectation
and conclude that there exists $r = r_t(\theta,\rho)$ such that,   for any fixed $t$,
\begin{align*}
\Ex [|I^\eta|] \le r_t(\theta,\rho)\to 0 \qquad \mbox{as $\theta, \rho \to 0$}.
\end{align*}
This leads to  estimate \eqref{eq:BV_necessity}
 and completes the proof.
\end{pf}

\smallskip
\begin{rem}\label{rem:whole_space}
In the whole-space case ({\it i.e.}, $\T^d$ is replaced by $\R^d$), it is necessary to modify the test function to
$\varphi(\xi, \zeta, \x, \y) = \eta_\rho''(\zeta-\xi) J_\theta(\y-\x) \psi(\frac{\y+\x}{2})$,
where $\psi$ is a non-negative smooth function $\R^d \to \R$ supported on $B_R(\mathbf{0})$.
The terms involving $\|\nabla \psi\|_{L^\infty}$ and $\|\psi\|_{L^1} $ appear respectively in the parabolic and It\^{o} correction terms.
In particular, in the It\^{o} correction and mollification estimates,
$\rho\|\psi\|_{L^1}$ is involved so that, for $R \to \infty$ (so that $\psi \to f(\x) \equiv 1$ pointwise),
one needs $\rho \to 0$ first; otherwise, no estimates are possible without prescribing very specific forms for $\psi$,
or we need to consider the weighted space $L^1(\psi(\x)\,\d \x)$.
\end{rem}

\section{$L^1$--Stability Estimate}\label{sec:I_L1_estimates}
In this section, we
establish the following $L^1$--stability theorem.

\begin{thm}[$L^1$--stability estimate]\label{thm:L1_stability}
Let $u$ and $v$ be kinetic solutions of \emph{\eqref{eq:equation1}}
with initial data $u_0$ and $v_0$, respectively.
Let the nonlinear functions of \emph{\eqref{eq:equation1}} satisfy
assumptions \emph{(\ref{eq:A_F})}--\emph{\eqref{eq:A_alpha}} with $\lambda_\sigma > \frac{1}{2}$.
Then the following $L^1$--stability estimate holds{\rm :}
\begin{equation}\label{4.1a}
\Ex\big[\int (v^+(\x,t) - u^+(\x,t))_+\;\d \x\big]
\leq  \exp\big\{C\|D_\x\cdot\FF_u\|_{L^\infty} t\big\}\,\Ex\big[\int (v_0(\x) - u_0(\x))_+\;\d \x\big],
\end{equation}
where $C$ is a constant depending only on $d$.
\end{thm}

\begin{pf}
In this case, $\FF(\cdot,\cdot) = \GG(\cdot,\cdot)$, $\A(\cdot) = \B(\cdot)$, and $\sigma(\cdot) = \tau(\cdot)$.
Then, from Proposition \ref{thm:flux_parabolic_bound}, we obtain
\begin{align*}
\Ex[I^a] \leq & \;C(\ab) \rho^{2\gamma_\ab}\theta^{-2} \Ex\big[\int_0^t \iint  \eta_\rho(v - u) \;\d\x\d\y\d s\big],\\
\Ex[I^F] \leq & \; C\|D_\x \cdot \FF_u\|_{L^\infty} \Ex\big[\int_0^t\iint \eta_\rho(v(\y)-u(\x)) J_\theta(\y-\x) \;\d\x\d\y\d s\big]\notag\\
&+ C\big(\rho^{\kappa_{F1}}\theta^{-1} + \theta^{\kappa_{F2}}\big)
  \Ex\big[\int_0^t\int \big(|(u, v)|^p+ |(u,v)|^q+1\big) \;\d\x\d s\big],\notag\\
\Ex[I^\sigma] \leq & \; C_\sigma  t \rho^{2\lambda_\sigma - 1},\\
\Ex[I^\eta] = &\, o_{\theta , \rho}(1) \to 0 \qquad\mbox{as $\theta,\rho\to 0$}.
\end{align*}
Taking the limits in the order: $\rho \to 0$ first and $\theta\to 0$ second, we obtain estimate:
\begin{equation}\label{4.1}
\begin{aligned}
&\Ex\big[\iint f^+(\xi,\x,t) \bar{g}^+(\xi,\x,t) \;\d \xi\,\d \x\big]\\
&\le\Ex\big[\int f^+(\xi,\x,0)\bar{g}^+(\xi,\x,0)\;\d \x\big]\\
&\quad\,\,+ C\|D_\x \cdot \FF_u\|_{L^\infty} \Ex\big[\int_0^t\iint f^+(\xi,\x,s) \bar{g}^+(\xi,\x,s) \;\d \xi\d \x\d s\big]
\end{aligned}
\end{equation}
for every $t \in [0,T]$.

Taking $f(\xi,\x,0) = H(\xi - u_0)$ and $\bar{g}(\zeta,\y,0) = \bar{H}(\zeta - v_0)$,
we can argue exactly as in \cite[Proposition~2.11]{DV2018} and conclude that the kinetic measure
does not concentrate at $t = 0$, $\mathbb{P}$-almost surely,
and that $f^+(\xi,\x,0) = f(\xi,\x,0)$ and $\bar{g}^+(\zeta,\y,0) = \bar{g}(\zeta, \y,0)$.

Taking $u_0 = v_0$ that leads to
$$
\Ex\big[\int f^+(\xi,\x,0)\bar{g}^+(\xi,\x,0)\;\d \x\big]=0,
$$
we see from \eqref{4.1} that
\begin{align*}
&\Ex\big[\iint f^+(\xi,\x,t) \bar{g}^+(\xi,\x,t) \;\d \xi\,\d \x\big]\\
&\leq   C\|D_\x \cdot \FF_u\|_{L^\infty}\int_0^t \Ex\big[\iint f^+(\xi,\x,s) \bar{g}^+(\xi,\x,s) \;\d \xi\d \x\big]\d s
\end{align*}
for every $t \in [0,T]$. Then it follows from the Gronwall inequality that
$$
\Ex\big[\iint f^+(\xi,\x,t) \bar{g}^+(\xi,\x,t) \;\d \xi\,\d \x\big]=0,
$$
which implies that,
for every $t \in [0,T]$,
$$
f^+(\xi,\x,t) \left(1 - f^+(\xi,\x,t) \right)= 0 \qquad\quad \mbox{$(\omega,\xi,\x)$-almost everywhere}.
$$
This means that, for every $t \in [0,T]$, $f^+$ takes values in $\{0,1\}$  $(\omega,\xi, \x)$-almost everywhere.
Similarly, it can be shown that, for every $t \in [0,T]$, $\bar{g}^+$ also takes values in $\{0,1\}$  $(\omega,\xi, \x)$-almost everywhere.
This allows us to represent $f$ and $\bar{g}$ as the Heaviside functions inside the integrals and to write
$$
\Ex\big[\iint f^+(\xi,\x,t) \bar{g}^+(\xi,\x,t) \;\d \xi\,\d \x\big]
= \Ex\big[\iint (v^+(\x,t) - u^+(\x,t))_+ \;\d \xi\,\d \x\big].
$$
Then bound \eqref{4.1a} follows directly from \ref{4.1} via the Gronwall inequality.
\end{pf}

\smallskip
It follows from Lemma \ref{thm:tech_lem} as in \cite[Corollary~3.3]{DV2018} (also see  \cite[Corollary~12]{DV2010} and  \cite[Corollary~3.4]{DHV2016}) that
\begin{cor}\label{thm:continuous_paths}
Let $u$ be a kinetic solution to \eqref{eq:eq1_parab}. There exists a version of $u$ with almost surely continuous paths in $L^p(\Td)$.
\end{cor}

In particular, from the proof of Proposition~\ref{thm:L1_stability} and Lemma~\ref{thm:tech_lem}, at each $t \in [0,T]$, $f^+$ can be represented $(\omega, \xi, \x)$--almost everywhere
as $H(\xi - u^+)$ and is right-continuous and, similarly, $f^-$ can be represented as $H(\xi - u^-)$, $(\omega, \xi,\x)$--almost everywhere
and is left-continuous.
From \eqref{3.12a}, it must be that $u^+ = u^-$ $\mathbb{P}$-almost surely, and $u$ is continuous $[0,T] \to L^p(\T^d)$.

From now on, we can drop the distinction between $u^\pm$ and $u$ (resp. $v^\pm$ and $v$) and simply refer to $u$ (resp. $v$);
we can also write $f^+(\xi,\x,t)$ as $H(\xi - u(\x,t))$ (resp. $\bar{g}^+(\zeta,\y,t)$ as $\bar{H}(\zeta - v(\y,t))$).

\smallskip
\begin{rem}
If $\FF$ is space-translational invariant (so that $D_\x\cdot\FF_u=0$), then we conclude
the familiar $L^1$--contraction estimate\,{\rm :}
\[
\Ex\big[\int (v(\x,t) - u(\x,t))_+ \;\d \x\big] \leq \Ex\big[\int (v_0(\x) - u_0(\x))_+\;\d \x\big].
\]
\end{rem}

\smallskip
\begin{cor}\label{thm:I_eta_refine}
With $I^\eta$ defined as in \eqref{3.13f}, we in fact have the bound{\rm :}
\begin{align}
\Ex[I^\eta] \le  \sup_{|\mathbf{h}| < \theta}\Ex\big[ \int (u(\x,t) - u(\x + \mathbf{h},t))_+\;\d\x\big].
\end{align}
\end{cor}

\begin{pf}
We can now write $I^\eta$ as
\begin{align*}
I^\eta &= \int (v(\x,t) - u(\x,t))_+\, \d \x  - \int H(\xi - u(\x,t)) \bar{H}(\zeta - v(\y,t)) \varphi(\xi,\zeta,\x,\y)\,\d E\\
& = \int (v(\x,t) - u(\x,t))_+\, \d \x  - \iint \eta_\rho(v(\y,t) -  u(\x,t)) J_\theta(\x - \y) \,\d \x\,\d \y.
\end{align*}

Using the basic inequality
$(\cdot)_+ \leq \eta_\rho(\cdot)$,
we obtain
\begin{equation*}
\begin{aligned}
&\Ex\big[\int (v(\x,t ) - u(\x,t ))_+ \;\d \x\big]\\
& = \Ex\big[\iint  (v(\y,t ) - u(\y,t ))_+ J_\theta(\y-\x) \;\d\x\d\y\big]\\
&\leq  \Ex\big[\iint (v(\y,t ) - u(\x,t ))_+  J_\theta(\y-\x) \;\d\x\d\y\big]\\
&\quad    + \Ex\big[\iint (u(\x,t ) - u(\y,t ))_+  J_\theta(\y-\x) \;\d\x\d\y\big]\\
& \leq  \Ex\big[\iint \eta_\rho(v(\y,t ) - u(\x,t ))  J_\theta(\y-\x) \;\d\x\d\y\big]\\
&\quad + \Ex\big[\iint (u(\x,t ) - u(\y,t ))_+  J_\theta(\y-\x) \;\d\x\d\y\big].
\end{aligned}
\end{equation*}
Set $\mathbf{h} := \y - \x$. Then we have
\begin{equation*}
\begin{aligned}
& \Ex\big[\iint (u(\x,t ) - u(\y,t ))_+  J_\theta(\y-\x) \;\d\x\d\y\big]\\
& = \Ex\big[\int \big(\int (u(\x,t ) - u(\x+\mathbf{h},t ))_+\, \d\x\big)J_\theta(\mathbf{h}) \;\d\mathbf{h}\big]\\
& \le \sup_{|\mathbf{h}| < \theta}
\Ex\big[\int (u(\x,t ) - u(\x + \mathbf{h},t ))_+ \;\d\x\big]\int J_\theta(\mathbf{h}) \;\d \mathbf{h}\\
& = \sup_{|\mathbf{h}| < \theta}\Ex\big[ \int (u(\x,t ) - u(\x + \mathbf{h},t ))_+ \;\d\x\big],
\end{aligned}
\end{equation*}
where we have used that $\displaystyle \int J_\theta(\mathbf{h}) \;\d \mathbf{h}=1$.
This completes the proof.
\end{pf}

\section{Fractional $BV$ Estimate} \label{sec:I_BV_estimate}

We now apply Proposition \ref{thm:flux_parabolic_bound} to the pair of two equations:
\begin{align}
\pd_t u =& - \nabla \cdot \FF(u,\x) + \nabla\cdot(\A(u) \nabla u) + \sigma(u) \dot{W},\label{5.1a}\\
\pd_t v =& - \nabla \cdot \FF(v,\x + \h) + \nabla\cdot(\A(v) \nabla v)+ \sigma(v) \dot{W},\label{5.2a}
\end{align}
with initial conditions $u(\x,0) = u_0(\x)$ and $v(\x,0) = u_0(\x + \h)$, respectively,
and derive a fractional $BV$ estimate.
In this case, $\gamma_\ab=\gamma_\be=\gamma$ and $\lambda_\sigma=\lambda_\tau=\lambda$.
With this fractional $BV$ estimate,
we can also refine our continuous dependence estimate.

\begin{thm}[Fractional $BV$ estimate]\label{thm:FBV_theorem}
Let $u$ be a kinetic solution of \emph{(\ref{eq:equation1})} with initial data $u_0$.
Let the nonlinear functions of \emph{(\ref{eq:equation1})} satisfy
assumptions \emph{(\ref{eq:A_F})}--\emph{(\ref{eq:A_alpha})} with $\lambda_\sigma >\frac{1}{2}$.
Then the following fractional $BV$ estimate holds{\rm :}
\begin{align}
&\Ex\big[\int (u(\x + \h,t) - u(\x,t))_+\;\d \x \big]\notag\\
&\leq  \exp\big\{C\|D_\x\cdot\FF\|_{L^\infty} t\big\}\notag\\
&\quad \times \Big(\hat{K}_1(\FF,t) |\h|^{\kappa_{F2}}
   + \Ex\big[\iint (u_0(\y + \h) - u_0(\x))_+J_{|\h|}(\y-\x)\;\d\x\d\y\big]\Big),\label{eq:BV_bound}
\end{align}
where $C$ depends on $d$, and $\hat{K}_1(\FF,t)$ depends on $\|(u_0,v_0)\|_{L^p}$
and is proportional to the H\"older norm of $D_\x \cdot \FF(\cdot, \x)$ in $\x$.

In particular, if $u_0$ is in the $\kappa_{F2}$--Nikolskii space with
$\kappa_{F2}\le 1$, i.e., the functions of bounded $\kappa_{F2}^{-1}$ variation,
then the fractional $BV$ bound holds{\rm :}
\begin{align*}
\Ex\big[|u|_{N^{\kappa_{F2},1}}(t)\big]
\leq \exp\big\{C\|D_\x\cdot\FF_u\|_{L^\infty} t\big\}\big(\hat{K}_1(\FF,t) + E\big[|u_0|_{N^{\kappa_{F2,1}}}\big]\big),
\end{align*}
where $|\cdot|_{N^{\kappa,1}}$ denotes the bounded $\frac{1}{\kappa}$--variation semi-norm,
the Nikolskii semi-norm \eqref{1.2b}.
\end{thm}

\begin{pf}
We first notice that, if $u(\x,t)$ solves \eqref{5.1a} with the initial data $u_0(\x)$,
then $v(\z,t)=u(\x+\h,t)$ for $\z=\x +\h$ solves \eqref{5.2a} with the initial data $u_0(\x+\h)$.

As in the $L^1$--stability estimate in \S 4,
choosing $\B = \A$ and $\tau = \sigma$ in Proposition \ref{thm:flux_parabolic_bound}
and using Corollary \ref{thm:I_eta_refine},
we have
\begin{align}
\Ex[I^a]\leq &\; C(\ab) \rho^{2\gamma}\theta^{-2}
   \Ex\big[\int_0^t \iint  \eta_\rho(v - u) J_\theta(\x - \y) \;\d\x\d\y\d s\big],\label{eq:I_rest-a}\\
\Ex[I^\sigma] \leq& \; C_\sigma\rho^{2\lambda - 1},\label{eq:I_rest-b}\\
\Ex[I^\eta] \leq &\;  \sup_{|\h|\leq \theta} \Ex\big[\int (u(\x + \h,t) - u(\x,t))_+ \;\d \x\big].\label{eq:I_rest-c}
\end{align}

Choosing $\GG(\cdot,\cdot) = \FF(\cdot, \cdot + \h)$,
we see from \eqref{eq:3.13c_expanded-a} and assumptions \eqref{eq:A_F}--\eqref{eq:A_kappa2} that
\begin{align*}
&\big|\FF_u(\xi,\x) \cdot \nabla_\x \varphi + \GG_u(\zeta, \y)\cdot \nabla_\y \varphi\big|\notag\\
& \le \left|\big(\FF_u(\zeta,\y + \h)  - \FF_u(\zeta, \y)\big)\cdot \nabla_\y \varphi \right|
  +\left|\big(\FF_u(\zeta,\y)  - \FF_u(\xi, \y)\big)\cdot \nabla_\y \varphi \right|\notag\\
&\quad\,\,+\left| \big(\FF_u(\xi,\y)  - \FF_u(\xi, \x)\big)\cdot \nabla_\y \varphi \right|
     + \big|\FF_u(\xi, \x)\cdot (\nabla_\x\varphi + \nabla_\y \varphi)\big|\\
&\le C \norm{D_\x \cdot \FF_u}_{L^\infty} |\h| \vartheta^{-1} \varphi + C |(\xi, \zeta)|^{p - 1} \rho^{\kappa_{F1}}\theta^{-1} \varphi
+ C\norm{D_\x \cdot \FF_u}_{L^\infty}\varphi,
\end{align*}
where we have used $|\nabla_\y \varphi| \le C \theta^{-1} \varphi$ and $\nabla_\x \varphi + \nabla _\y \varphi = 0$.

Furthermore,  using \eqref{eq:FF_est2},
\begin{align*}
& \left|\int H(\xi - u) \bar{H}(\zeta - v) \big(D_\y\cdot\GG(\zeta,\y) - D_\x\cdot\FF(\xi,\x)\big) \varphi_\zeta \,\d E\right|\notag\\
&\le\left|\int H(\xi - u) \bar{H}(\zeta - v) \big(D_\y\cdot\FF(\zeta,\y + \h) - D_\x\cdot\FF(\xi,\y + \h)\big) \varphi_\zeta \,\d E\right|\\
&\quad\,\, + \left|\int H(\xi - u) \bar{H}(\zeta - v) \big(D_\y\cdot\FF(\xi ,\y + \h) - D_\x\cdot\FF(\xi,\x)\big) \varphi_\zeta \,\d E\right|\\
& \le C \norm{D_\x \cdot \FF}_{L^\infty}  \left| \int  H(\xi - u) \bar{H}(\zeta - v) \varphi(\xi,\zeta,\x,\y) \,\d E\right|\\
&\quad\, + C (\theta + |\h|)^{\kappa_{F2}}\int \big(|(u,v)|^q + 1\big)\;\d\x.
\end{align*}

Let $|\h|, \theta < 1$. Then
\begin{align*}
|I^F| \leq &\; C\|D_\x \cdot \FF_u\|_{L^\infty} \int_0^t\iint \eta_\rho( v(\y,s) - u(\x,s)) \big(|\h|\theta^{-1}+1\big)
 J_\theta(\y - \x) \;\d\x\d\y\d s\notag\\
&\; + C\big(\rho^{\kappa_{F1}}\theta^{-1} + (\theta + |\h|)^{\kappa_{F2}}\big) \int_0^t\int \big(|(u,v)|^p + |(u,v)|^q+1\big) \;\d\x\d s.
\end{align*}
Combining this estimate with \eqref{eq:I_rest-a}--\eqref{eq:I_rest-c}, we have
\begin{align*}
&\Ex\big[\iint \eta_\rho(u(\y + \h, t) - u(\x, t)) J_\theta(\y - \x) \;\d\x\d\y \big]\\
&\leq \Ex\big[\iint \eta_\rho(u_0(\y + \h) - u_0(\x))J_\theta(\y - \x)\;\d\x\d\y \big]\\
&\quad + C\|D_\x\cdot\FF_u\|_{L^\infty}\big(|\h|\theta^{-1} + 1\big)\Ex\big[\int_0^t \iint  \eta_\rho(v - u) J_\theta(\y - \x)\;\d\x\d\y\d s\big] \\
&\quad + C\rho^{2\gamma}\theta^{-2}\Ex\big[\int_0^t \iint  \eta_\rho(v - u) J_\theta(\y - \x)\;\d\x\d\y\d s\big]\\
&\quad + C\big(\rho^{\kappa_{F1}}\theta^{-1} + (\theta + |\h|)^{\kappa_{F2}}\big) \Ex\big[\int_0^t\int \big(|(u,v)|^p + |(u,v)|^q+1\big)
  \;\d\x\d s\big]\\
&\quad  + C_\sigma\rho^{2\lambda - 1}  |\T^d|.
\end{align*}
Next, we apply the Gronwall inequality and use the estimates on $\Ex\big[I^\eta]$ to conclude
the proof by choosing $\theta = |\h|$ and taking $\rho \to 0$.

In particular, if $u(\x,t)$ is in the fractional $BV$ class in $\x$ with index $\kappa$ for any fixed $t>0$, then
$$
\Ex\big[\int \big(u(\x+\h,t) - u(\x,t)\big)_+ \;\d \x\big] \le C |\h|^{\kappa},
$$
which is equivalent to
$$
\Ex\big[\int \big(u(\y+\h,t) - u(\x,t)\big)_+ J_{|\h|}(\x - \y)\;\d \x\d\y\big] \le C |\h|^{\kappa}.
$$
This can be seen as follows: If $\displaystyle \Ex\big[\int \big(u(\x+\h,t) - u(\x,t)\big)_+ \;\d \x\big] \le C |\h|^{\kappa}$,
then
\begin{align*}
&\Ex\big[\iint (u(\y+\h,t) - u(\x,t))_+ J_{|\h|}(\y - \x)\;\d\x\d\y\big]\\
&\leq \Ex\big[\iint \big((u(\y+\h,t) - u(\x + \h,t))_+ \\
 &\qquad\qquad\,\,\, + (u(\x + \h,t)- u(\x,t))_+\big) J_{|\h|}(\y - \x)\;\d\x\d\y\big]\\
&\leq \Ex\big[\iint \big((u(\y,t) - u(\x,t ))_+ J_{|\h|}(\y - \x) \\
 &\qquad\qquad\,\,\,   + (u(\x + \h,t)- u(\x,t))_+ J_{|\h|}(\y - \x)\big)\;\d\x\d\y\big]\\
&\leq   \sup_{\mathbf{z} \in B_{|\h|}(\mathbf{0})} \Ex\big[\int (u(\x + \mathbf{z},t) - u(\x,t))_+\;\d \x \;\int J_{|\h|}(\y - \x)\;\d\y\big]
  + C|\h|^{\kappa}\leq C|\h|^{\kappa}.
\end{align*}
Conversely, if $\displaystyle \Ex\big[\iint (u(\y+\h,t) - u(\x,t))_+ J_{|\h|}(\y - \x)\;\d \x\d\y\big] \le C |\h|^{\kappa}$, then
\begin{align*}
&\Ex\big[\int (u(\x+\h,t) - u(\x,t))_+ \;\d \x\big]\\
&\leq  \Ex\big[\iint  \big((u(\x+\h,t) - u(\y,t) )_+J_\theta(\y - \x)\\
&\qquad\qquad\,\,\, + (u(\y,t) - u(\x,t))_+ J_\theta(\y - \x)\big) \;\d\x\d\y\big]\\
&\leq  C|\h|^\kappa + \sup_{\mathbf{z}\in B_{|\h|}(\mathbf{0})} \Ex\big[\iint (u(\y + \mathbf{z},t) - u(\x,t))_+ J_\theta(\y - \x) \;\d\x\d\y\big]\leq  C |\h|^\kappa.
\end{align*}

Therefore, if $\Ex\big[|u_0|_{N^{\kappa,1}}\big]< \infty$, then $\Ex\big[|u|_{N^{\kappa,1}}\big]<\infty$.
This completes the proof.
\end{pf}

\begin{rem}
If $\kappa_{F2} = 1$, we obtain an actual $BV$ estimate by taking
the supremum ({\it cf.} \cite[Theorem 1.7.2]{Daf2016} and \cite[Definition 1]{Sim1990} for the deterministic case),
whilst sending $\theta = |\h| \to 0$.
In fact, adding to the inequality by the corresponding inequality for $\big(u(\y + \h) - u(\x)\big)_-$, we have
\begin{align}\label{eq:BV_bound2}
\Ex\big[|u|_{BV}(t)\big]
\leq  \exp\big\{C(d)\|D_\x\cdot\FF_u\|_{L^\infty} t\big\}\big(\hat{K}_1(\FF,t) + \Ex\big[|u_0|_{BV}\big]\big).
\end{align}
Finally, in the space-translational invariant case,  $\hat{K}_1(\FF) = 0$ and $\|D_\x\cdot\FF_u\|_{L^\infty} = 0$,
so that the classical $BV$ bound follows:
\begin{align}
\Ex\big[|u|_{N^{\kappa_{F2},1}}(t)\big] \leq \Ex\big[|u_0|_{N^{\kappa_{F2},1}}\big].
\end{align}
In particular, when $\kappa_{F2}=1$,
\begin{align}
\Ex\big[|u|_{BV}(t)\big] \leq \Ex\big[|u_0|_{BV}\big].
\end{align}
\end{rem}

\medskip
\section{Continuous Dependence Estimate}\label{sec:I_continuous_dependence}

A continuous dependence estimate for equations \eqref{eq:equation1}--\eqref{eq:equation2}
is an estimate of form:
\[
\Ex[\|v(\cdot,t) - u(\cdot,t)\|] \leq C(\A,\B,\FF,\GG,\sigma, \tau,u_0,v_0,t)M(\B - \A ,\GG -  \FF, \tau - \sigma,v_0 -  u_0 ,t),
\]
where $M$ tends to zero as the arguments $(\B - \A, \GG -\FF, \tau - \sigma, v_0 - u_0)$ tend to zero,
and $\|\cdot\|$ is a norm or  semi-norm.

To prove the full continuous dependence estimates,
we use our (fractional) $BV$ estimates to refine both the mollification estimates (\ref{eq:BV_necessity})
and the estimates in Proposition \ref{thm:flux_parabolic_bound}.

\begin{thm}\label{thm:continuous_dependence}
Let $u$ be a kinetic solution of \eqref{eq:equation1}
on $\T^d$ with initial data $u_0 \in N^{\kappa,1} \cap L^p$ for $\kappa \geq \kappa_{F2}$.
Let $v$ be a kinetic solution of \eqref{eq:equation2}
on $\T^d$ with initial data $v_0\in L^p$.
Assume that $\FF$ and $\GG$ satisfy \emph{(\ref{eq:A_F})}--\emph{(\ref{eq:A_kappa2})}
and \emph{(\ref{eq:AB_F})}--\emph{(\ref{eq:AB_kappa2})}, respectively.
Let $\sigma$ and $\tau$ satisfy \emph{(\ref{eq:A_sigma})}
and {\em \eqref{eq:AB_sigma}} with $\lambda_\sigma,\lambda_\tau>\frac{1}{2}$,
and let $\A$ and $\B$ satisfy \emph{(\ref{eq:A_alpha})} and {\em \eqref{eq:AB_alpha}}
with $\gamma_{\ab}, \gamma_{\be} > \frac{1}{2}$, respectively.
For any real constants $\rho, \theta>0$, the following continuous dependence estimate holds{\rm :}
\begin{align*}
&\Ex\big[\int (v(\x,t) - u(\x,t))_+\;\d \x\big] \notag\\
&\leq C \rho + C\theta^{\kappa_{F2}} \exp\big\{C\|D_\x\cdot\FF\|_{L^\infty}t \big\}\big(\hat{K}_1(\FF,t) + \Ex\big[|u_0|_{N^{\kappa,1}}\big]\big)\\
&\quad +\exp\{\mathcal{L}t\} \Big(\Ex\big[\iint \eta_\rho(v(\y,0) - u(\x,0)) J_\theta(\y - \x)\;\d\x\d\y\big]  \\
&\qquad\qquad\qquad\, +\big(\rho^{\kappa_{F1}} \theta^{-1} + \theta^{\kappa_{F2}}\big)\hat{K}(u_0,v_0,t)
 +  C t\rho^{-1}\big(\rho^{2\lambda_\sigma} + \|\tau - \sigma\|^2_{L^\infty}\big)\Big),
\end{align*}
where
\begin{align}
\mathcal{L} = &\, C(\ab) \big(\|\sqrt{\B} - \sqrt{\A}\|_{L^\infty}^2 +\rho^{2\gamma_\ab}\big)\theta^{-2}\notag\\
&\, +  C\big(\|\GG_u - \FF_u\|_{L^\infty} \theta^{-1}+\|D_\x \cdot (\GG - \FF)\|_{L^\infty} \rho^{-1}\big)+ C\|D_\x \cdot \FF_u\|_{L^\infty},
\label{6.1a}
\end{align}
with all the constituent differences assumed to be bounded,
 \begin{align*}
 \hat{K}(u_0,v_0, t) &= E\big[\int_0^t \big(\|(u, v)(\cdot,s)\|_{L^p}^p + 1\big) \;\d s\big]\\
   &\leq \exp\big\{C_0T\,\Ex[\|(u_0,v_0)\|_{L^p}^p]\big\} \qquad \mbox{for $t\in [0, T]$},
 \end{align*}
 and $C_0$ is a constant depending on $\FF, \GG, \A, \B,\sigma, \tau, d, T$, and $|\mathbb{T}^d|$.

In particular, for $u_0 \in BV \cap L^p$ and $\kappa_{F2}=1$, we can choose $\mu < \kappa_{F1}$ and set
\begin{align}
\rho^\mu = \theta
=t^{\frac{\mu}{2}}\big( \big\|(\GG_u - \FF_u, \sqrt{\B} - \sqrt{\A})\big\|_{L^\infty}
  + \big\|(\tau - \sigma, D_\x\cdot  (\GG - \FF))\big\|_{L^\infty}^\mu\big)
\label{6.2}
\end{align}
to yield that there exists  a constant $C>0$, depending on $T>0$, such that
\begin{align*}
&\Ex \big[\int (v(\x,t) - u(\x,t))_+\;\d \x\big] \notag\\
&\leq C
 \Ex\big[\int (v_0(\x) - u_0(\x))_+\;\d \x\big] \\
&\quad
 +C\Big(\big\|(\GG_u - \FF_u, \sqrt{\B} - \sqrt{\A})\big\|_{L^\infty}+
   \big\|(\tau - \sigma,D_\x\cdot (\GG - \FF))\big\|_{L^\infty}^\mu\Big)^r
\end{align*}
for any $0 < \mu < \kappa_{F1}$,
where $r:=\min\{\frac{\kappa_{F1}}{\mu}-1, \, \frac{2\lambda_\sigma - 1}{\mu},\,1,\,\frac{1}\mu\}$.
\end{thm}

\begin{pf} We divide the proof into three steps.

\smallskip
1. {\it Refinement of the mollification estimate}:
With the assumption that $u_0\in N^{\kappa,1}$, we return to the mollification estimate \eqref{eq:BV_necessity} and \eqref{eq:BV_bound}:
\begin{align}
\Ex[I^\eta]
\leq C\theta^{\kappa_{F2}} \exp\big\{C\|D_\x\cdot\FF\|_{L^\infty}t \big\}\big(\hat{K}_1(\FF,t) + \Ex\big[|u_0|_{N^{\kappa_{F2},1}}\big]\big).
\label{eq:BV_necessity2}
\end{align}
Moreover, when $t = 0$,
\begin{align*}
&\Ex\big[\big|\iint \eta_\rho(v_0(\y) - u_0(\x)) J_\theta(\y - \x)
\;\d \x\,\d \y - \int (v_0(\y) - u_0(\y))_+\;\d \y\big|\big]\\
&\leq  C \rho + C\theta^{\kappa_{F2}} \Ex\big[|u_0|_{N^{\kappa_{F2},1}}\big].
\end{align*}

\smallskip
2. {\it Continuous dependence estimate}:
We now prove the general continuous dependence estimate for the initial data in $N^{\kappa_{F2},1}$. From
Proposition \ref{thm:flux_parabolic_bound}, we have the estimates:
\begin{align*}
&\Ex\big[\iint \eta_\rho(v(\y, t) - u(\x,t)) J_\theta(\y - \x)\;\d\x\d\y\big] \\
&\leq \Ex\big[\iint \eta_\rho(v_0(\y) - u_0(\x)) J_\theta(\y - \x)\;\d\x\d\y\big]\\
&\quad + C\|D_\x \cdot \FF_u\|_{L^\infty}\Ex\big[\int_0^t\iint \eta_\rho(v - u) J_\theta(\y - \x)\;\d\x\d\y\d s\big]\\
&\quad + \Big(C(\ab)\theta^{-2} \big(\|\sqrt{\B} - \sqrt{\A}\|_{L^\infty}^2 +\rho^{2\gamma_\ab}\big)\\
&\qquad\,\,\,\,\,    + C\big(\|\GG_u - \FF_u\|_{L^\infty} \theta^{-1}
  + \|D_\x \cdot (\GG - \FF)\|_{L^\infty} \rho^{-1}\big)\Big)\\
&\qquad\qquad \, \times   \Ex\big[\int_0^t \iint \eta_\rho(v - u) J_\theta(\y - \x) \;\d\x\d\y\d s\big]\\
&\quad  + Ct \big(\rho^{-1}\|\sigma - \tau\|_{L^\infty}^2 + \rho^{2\lambda_\sigma - 1}\big)|\mathbb{T}^d|
  +\big(\rho^{\kappa_{F1}} \theta^{-1} + \theta^{\kappa_{F2}}\big)\hat{K}(u_0,v_0,t).
\end{align*}

Applying the Gronwall inequality to the preceding calculation, we have
\begin{align*}
&\Ex\big[\iint\eta_\rho(v(\y,t) - u(\x,t)) J_\theta(\y - \x)\;\d\x\d\y\big] \notag\\
&\leq e^{\mathcal{L} t} \Big(\Ex\big[\iint \eta_\rho(v(\y,0) - u(\x,0)) J_\theta(\y - \x)\;\d\x\d\y\big]  \\
&\qquad\,\,\,\,\,\,\, + \big(\rho^{\kappa_{F1}} \theta^{-1} + \theta^{\kappa_{F2}}\big)\hat{K}(u_0,v_0,t)
 +  C t\big(\rho^{2\lambda_\sigma - 1} + \rho^{-1}\|\tau - \sigma\|^2_{L^\infty}\big)\Big),
\end{align*}
where $\mathcal{L}$ is defined by \eqref{6.1a}.
Now we apply the mollification estimate (\ref{eq:BV_necessity2}) to obtain
\begin{align*}
&\Ex\big[\int (v(\x,t) - u(\x,t))_+\;\d \x\big] \notag\\
&\leq C \rho + C\theta^{\kappa_{F2}} \exp\big\{C\|D_\x\cdot\FF\|_{L^\infty}t \big\}\big(\hat{K}_1(\FF,t) + \Ex\big[|u_0|_{N^{\kappa_{F2},1}}\big]\big)\\
&\quad +\exp\{\mathcal{L}t\} \Big(\Ex\big[\iint\eta_\rho(v(\y,0) - u(\x,0)) J_\theta(\y - \x)\;\d\x\d\y\big]  \\
&\qquad\qquad\quad\quad\,\, + C\big(\rho^{\kappa_{F1}} \theta^{-1} + \theta^{\kappa_{F2}}\big)\hat{K}(u_0,v_0,t)
 +  C t \rho^{-1}\big(\rho^{2\lambda_\sigma} +\| \tau -\sigma\|^2_{L^\infty}\big)\Big).
\end{align*}

\smallskip
3. {\it Refinement of the continuous dependence estimate}:
Next, we consider the $BV$ case.
Assuming that $\kappa_{F2} = 1$, we can further refine the estimate.

Since $\Ex\big[|u(\cdot, t)|_{BV}\big]$ is now bounded, we can refine the estimates in
Proposition \ref{thm:flux_parabolic_bound}.
Let $P \in L^\infty$ be some generic placeholder.
Then integrating by parts yields
\begin{align*}
&\left|\int H(\xi - u) \bar{H}(\zeta - v) P(\xi,\zeta) \cdot \nabla_\x \varphi(\xi,\zeta,\x,\y)\;\d E\right|\\
&= \left|\int H(\xi - u) \bar{H}(\zeta - v) P(\xi,\zeta) \cdot \nabla_\x J_\theta(\y - \x)  \eta''_\rho(\zeta-\xi)\;\d E\right|\\
&= \left| \iint  \int  \bar{H}(\zeta - v)\nabla_\x u\cdot P(u,\zeta) J_\theta(\y - u)\eta''_\rho(\zeta-\xi)
  \; \d\zeta\d\x\d\y \right|\\
&\leq  \|P\|_{L^\infty} \iint \eta'_\rho(v - u) J_\theta(\y - \x)  |\nabla_\x u| \;\d\x\d\y\\
&\leq |\mathbb{T}^d| \|P\|_{L^\infty}  |u(t)|_{BV},
\end{align*}
where we have used $\eta_\rho''\ge 0$, $J_\theta\ge 0$, and the boundedness of $\eta'_\rho$.
This means that
\[
\theta^{-1} \iint \eta_\rho (v - u) J_\theta(\y - \x)\;\d\x\d\y
\]
in (\ref{eq:rev3})--(\ref{eq:rev5}) can be replaced by $|\mathbb{T}^d| |u(t)|_{BV}$
to avoid an application of the Gronwall inequality (which puts $\theta^{-1}$ in an exponent)
and an exponential penalization in time
here (which comes from estimate \eqref{eq:BV_bound} on $|u(t)|_{N^{\kappa_{F2},1}}$ instead).

In particular, we have
\begin{align}
&\left|\int_0^t \int  \bar{H}(\zeta - v) H(\xi - u) \bk{\be(\zeta) -\ab(\zeta) }\bk{\be(\zeta) - \ab(\zeta) } :\nabla^2_{\x} \varphi\;\d E\d s\right|\notag\\
&\quad \leq  d\big\|\sqrt{\B} - \sqrt{\A}\big\|^2_{L^\infty} \theta^{-1}\int_0^t |u(\cdot,s)|_{BV}\;\d s, \label{eq:BV_refinement3}\\[2mm]
& \int_0^t \int  \bar{H}(\zeta - v) H(\xi - u)  \bk{\ab(\zeta) - \ab(\xi) }\bk{\be(\zeta) - \ab(\zeta) }:\nabla^2_{\x} \varphi\;\d E \d s\notag\\
&\quad \leq  C(\ab,d)\big\|\sqrt{\B} - \sqrt{\A}\big\|_{L^\infty}\rho^{\gamma_\ab} \theta^{-1}\int_0^t |u(\cdot,s)|_{BV}\;\d s,\label{eq:BV_refinement4}\\[2mm]
& \int_0^t \int  \bar{H}(\zeta - v) H(\xi - u)  \bk{\ab(\xi) - \ab(\zeta)}\bk{\ab(\xi) - \ab(\zeta)}:\nabla^2_{\x} \varphi\;\d E\d s\notag\\
&\quad \leq  C(\ab,d) \rho^{2\gamma_\ab} \theta^{-1}\int_0^t |u(\cdot,s)|_{BV}\;\d s, \label{eq:BV_refinement5}
\end{align}
in place of (\ref{eq:rev3})--(\ref{eq:rev5}).

Similarly, we have
\begin{align}
 &\left|\int_0^t \int  \bar{H}(\zeta - v)  H(\xi - u)\big(\GG_u(\xi,\x)  - \FF_u(\xi, \x)\big)\cdot \nabla_\x \varphi \;\d E\d s\right|\notag\\
 &\leq C\|\GG_u-\FF_u\|_{L^\infty} \int_0^t |u(\cdot,s)|_{BV}\;\d s,\label{eq:BV_refinement1}
 \end{align}
in place of \eqref{eq:rev1}.

As in Step 2 above,
using Proposition \ref{thm:flux_parabolic_bound}, we arrive at the bound:
\begin{align*}
&\Ex\big[\iint \eta_\rho(v(\y, t) - u(\x,t)) J_\theta(\y - \x)\;\d\x\d\y\big] \\
&\leq \Ex\big[\iint \eta_\rho(v_0(\y) - u_0(\x)) J_\theta(\y - \x)\;\d\x\d\y\big]\\
&\quad + C\|D_\x \cdot \FF_u\|_{L^\infty}\Ex\big[\int_0^t\iint \eta_\rho(v - u) J_\theta(\y - \x)\;\d\x\d\y\d s\big]\\
&\quad +  C(\ab) \big(\|\sqrt{\B} - \sqrt{\A}\|_{L^\infty}^2 \theta^{-1}
   + \rho^{2\gamma_\ab} \theta^{-1}  + \|\GG_u - \FF_u\|_{L^\infty} \big) \Ex\big[\int_0^t |u(\cdot,s)|_{BV} \,\d s\big]\\
&\quad  + C\|D_\x \cdot (\GG - \FF)\|_{L^\infty} \rho^{-1} \Ex\big[\int_0^t \iint \eta_\rho(v - u) J_\theta(\y - \x) \;\d\x\d\y\d s\big]\\
&\quad  + Ct \big(\rho^{-1}\|\tau - \sigma\|_{L^\infty}^2 + \rho^{2\lambda_\sigma - 1}\big)|\mathbb{T}^d|
  + C\big(\rho^{\kappa_{F1}} \theta^{-1} + \theta^{\kappa_{F2}}
  \big)\hat{K}(u_0,v_0,t).
\end{align*}
Estimating the mollification and $BV$ terms, we derive a continuous dependence estimate for
$$
 \Ex\big[\iint \eta_\rho(v(\x, t) - u(\x,t))\;\d\x\big]
$$
as before.
Since $2\gamma_\ab >1$, we can choose $\mu < \kappa_{F1}$ and set $\rho$ and $\theta$ as in \eqref{6.2} to complete the proof.
\end{pf}

\begin{rem}
Estimate \eqref{eq:BV_refinement1} can also be applied to \eqref{eq:FF_est1} if $\FF_u(\xi,\x) - \FF_u(\zeta, \x)$ is assumed
to be uniformly bounded, replacing $\theta^{-1}$ by $\int_0^t |u(\cdot, s)|_{BV}\;\d s$ in  \eqref{eq:FF_est1}.
\end{rem}

\begin{rem} The case that  $\A$ depends on $\x$, {\it i.e.}, $\A=\A(u,\x)$, behaves differently,
and additional difficulties present themselves.
In particular, for the $BV$--estimate, in order to make a sense of the calculations,
one might take the $i$th derivative of the entire equation (at the bulk, non-kinetic level) and test it against $\eta'_\rho(\pd_i u)$.
One cannot easily propose an assumption on $\A(u,\x)_{x_i}$ by which to bound the terms:
\begin{align*}
\int \eta''_\rho(u_{x_i}) \A(u,\x)_{x_i}: (\nabla u \otimes \nabla u_{x_i}) \;\d \x,
\end{align*}
since the second derivatives inevitably appear in the estimates.
\end{rem}

\section{Existence of Stochastic Kinetic Solutions}\label{sec:I_existence}

In this section, we employ the continuous dependence estimate to establish the existence
of stochastic kinetic solutions.
In order to achieve proper energy estimates in $L^p(\mathbb{T}^d)$, we require
the assumption that
\begin{align*}
|D_\x \cdot \FF (u,\x)|\leq C \big(|u|+1\big).
\end{align*}

\begin{rem}\label{rem:whole_space_d}
With reference to Remark \ref{rem:whole_space},
it is possible to extend this result to $L^p(\mathbb{R}^d)$, since only the $L^1$--stability is actually used.
\end{rem}

\subsection{Convergence in $\ep$}

Let $u_0^\ep$ be a collection of the initial data functions that tend to $u_0$ in $L^1_\omega L_\x^1$.

We show here that there is a subsequence of the corresponding viscosity kinetic solutions $u^\ep$
(see
\ref{sec:lwp_parabolic}
for the well-posedness of solutions with almost surely continuous paths in $L^p(\T^d)$),
which converges to a unique stochastic kinetic solution.
From the continuous dependence estimates, we conclude that the kinetic solutions $u^\ep$ and $u^{\ep'}$ of
\begin{align*}
\pd_t u^\ep = -\nabla \cdot \FF(u^\ep, \x) + \nabla \cdot \big((\A(u^\ep) + \ep\mathbf{I}) \nabla u^\ep\big) + \sigma(u^\ep) \dot{W},
\end{align*}
and
\begin{align*}
\pd_t u^{\ep'} = -\nabla \cdot \FF(u^{\ep'}, \x) + \nabla \cdot\big((\A(u^{\ep'}) + \ep'\mathbf{I}) \nabla u^{\ep'}\big) + \sigma(u^{\ep'}) \dot{W},
\end{align*}
satisfy
\begin{align*}
&\Ex \big[\int |u^{\ep'}(\x,t) - u^\ep(\x,t)|\;\d \x\big] \notag\\
&\leq C
\Big(\Ex\big[\int |u^{\ep'}_0(\x) - u_0^\ep(\x)|\;\d \x\big]
+ |\sqrt{\ep} - \sqrt{\ep'}|^{\min\{\frac{\kappa_{F1}}{\mu}-1,\,\frac{2\lambda_\sigma-1}{\mu},\,\kappa_{F2},\, \frac{1}{\mu}\}}\Big)
\end{align*}
for $0<\mu<\kappa_{F1}$.
Then we conclude that
\begin{align}\label{eq:L1_tx}
\Ex\big[\|u^{\ep'}(\x,t) - u^\ep(\x,t)\|_{L^1(\mathbb{T}^d\times [0,T])}\big] \to 0 \qquad\mbox{as $\ep, \ep'\to 0$}.
\end{align}

Moreover, including the martingale part in order to estimate the difference in the uniform norm in time, we have
\begin{align}
&\Ex \big[\sup_{t \in [0,T]}\int |u^{\ep'}(\x,t) - u^\ep(\x,t)|\;\d \x\big] \nonumber\\
&\leq C
\Big(\Ex\big[\int |u^{\ep'}_0(\x) - u_0^\ep(\x)|\;\d \x\big]
+ |\sqrt{\ep} - \sqrt{\ep'}|^{\min\{\frac{\kappa_{F1}}{\mu}-1,\,\frac{2\lambda_\sigma-1}{\mu},\,\kappa_{F2},\, \frac{1}{\mu}\}}\Big)\nonumber\\
&\quad\,\, + \Ex \big[\sup_{t \in [0,T]} \Big|\int_0^t \int \big(\sigma(u^\ep) - \sigma(u^{\ep'})\big)\,\d \x \,\d W\Big|\big].\label{eq:convergenceL1_xC_t}
\end{align}

By the Burkholder--Davis--Gundy inequality and Young's inequality,
\begin{align*}
& \Ex \big[ \sup_{t \in [0,T]} \left|\int_0^t \int \big(\sigma(u^\ep) - \sigma(u^{\ep'})\big)\,\d \x \,\d W\right|\big]\\
 & C \le  \Ex \big[\Big|\int_0^T \int\frac{ \big|\sigma(u^\ep) - \sigma(u^{\ep'})\big|^2}{|u^\ep - u^{\ep'}|^{2\lambda_\sigma}}
   |u^\ep - u^{\ep'}|^{2\lambda_\sigma}\,\d \x \,\d s\Big|^{1/2}\big]\\
 & \le C_\sigma \Ex\big[\Big|\int_0^T \int |u^\ep - u^{\ep'}|^{2\lambda_\sigma}\,\d \x \,\d s\Big|^{1/2}\big]\\
 & \le C_\sigma \Ex  \big[\int_0^T \int |u^\ep - u^{\ep'}|^{2\lambda_\sigma - 1}\,\d \x \,\d s\big] + \frac12 \Ex \big[\sup_{t \in [0,T]} \int |u^\ep - u^{\ep'}| \,\d \x\big]\\
  & \le C_{\sigma, T} \Big(\Ex  \big[\int_0^T \int |u^\ep - u^{\ep'}|\,\d \x \,\d s\big]\Big)^{2\lambda_\sigma - 1}
  + \frac12 \Ex \big[\sup_{t \in [0,T]} \int |u^\ep - u^{\ep'}| \,\d \x\big].
\end{align*}
The final inequality is the result of Jensen's inequality, as $2 \lambda_\sigma  - 1 \in (0,1)$.

Since $2 \lambda_\sigma - 1 > 0$,
from \eqref{eq:convergenceL1_xC_t},
we have the following bound:
\begin{align*}
&\frac12 \Ex \big[\sup_{t \in [0,T]}\int |u^{\ep'}(\x,t) - u^\ep(\x,t)|\;\d \x\big] \notag\\
&\leq C
\Big(\Ex\big[\int |u^{\ep'}_0(\x) - u_0^\ep(\x)|\;\d \x\big]
+ |\sqrt{\ep} - \sqrt{\ep'}|^{\min\{\frac{\kappa_{F1}}{\mu}-1,\,\frac{2\lambda_\sigma-1}{\mu},\,\kappa_{F2},\, \frac{1}{\mu}\}}\Big)\notag\\
&\quad\,  + f(\ep,\ep'),
\end{align*}
where
$$
f(\ep,\ep') :=C_{\sigma, T} \Big(\Ex\big[\int_0^T \int |u^\ep - u^{\ep'}|\,\d \x \,\d s\big]\Big)^{2\lambda_\sigma - 1} \to 0
$$
as $\ep,\ep' \to 0$, by \eqref{eq:L1_tx}.

Finally, we conclude that
\begin{align*}
\Ex\big[\sup_{t \in[0,T]}\|u^{\ep'}(\x,t) - u^\ep(\x,t)\|_{L^1(\mathbb{T}^d)}\big] \to 0 \qquad\mbox{as $\ep, \ep'\to 0$}.
\end{align*}
That is,
the approximate solution sequence $\{u^\ep\}$ is a Cauchy sequence in $L^1_\omega C_t L^1_\x$ so that there is a
subsequence (still denoted) $\{u^\ep\}$ converging to a process $u$ with almost surely continuous paths in $L^1(\T^d)$.

\subsection{Existence Theorem}
With the convergence of $\{u^\ep\}$ obtained in \S 7.1, we can
follow \cite{CDK2012} to conclude the  following existence theorem:

\begin{thm}[Existence of stochastic kinetic solutions]
Let assumptions \emph{(\ref{eq:A_F})}--\emph{(\ref{eq:A_alpha})} hold.
Then there exists a unique stochastic kinetic solution of
equation \eqref{eq:eq1_parab} with initial data $u_0\in L^1$.
In particular, if initial data $u_0\in L^p \cap N^{\kappa,1}$,
then the stochastic kinetic solution $u(\cdot, t)\in L^p \cap N^{\kappa,1}$
for each $t>0$.
\end{thm}

\begin{pf}
For any fixed $\ep$, we can mollify $u_0$ into $u_0^\ep \in C^\infty$ so that
$\Ex[\|u_0^\ep\|_{H^s}^2]$ is bounded for any $s$ and
\[
\Ex\big[\|u_0^\ep\|_{L^p}^p\big]
\leq C\,\Ex\big[\|u_0\|_{L^p}^p\big]<\infty,
\]
where $C>0$ is a constant independent of $\ep>0$.
Then, as in \cite{CDK2012}, using the arguments of \S 4 of Feng--Nualart \cite{FN2008},
together with the convergence results in \S 7.1, we can conclude that
there is a convergent subsequence $u^\ep(\x,t)$ that converges {\it a.e.} almost surely to $u(\x,t)$
that is a stochastic kinetic solution. The $L^1$--stability of stochastic kinetic solutions implies the
uniqueness of the solution.

\medskip
In particular, if $\Ex\big[\|u_0\|_{L^p}^p\big]  + \Ex\big[|u_0|_{N^{\kappa,1}}\big] < \infty$,
by the continuous dependence estimates, we conclude that
$$
\sup_{t>0}\big(\Ex\big[\|u(\cdot, t)\|_{L^p}^p\big]+\Ex\big[|u(\cdot,t)|_{N^{\kappa,1}}\big]\big)<\infty.
$$
\end{pf}

\section{Temporal Fractional $BV$ Regularity of Stochastic Kinetic Solutions}

In this section, we prove that the stochastic kinetic solution is
of fractional $BV$ regularity in time.

\begin{thm} Let $u_0(\cdot)\in L^p \cap N^{\kappa_1,1}$ for some $\kappa_1>0$.
Let $D_\x\cdot\FF$  have linear growth in $u$ and $\kappa_2$--H\"{o}lder in $\x$ for some $\kappa_2>0$.
Let $\sigma$ have linear growth, and let the entries of $\A$ have polynomial growth in $u$.
Then there exists $\beta>0$ depending on $\kappa_1$ and $\kappa_2$ such that, for any $T>0$, there is $C_T>0$ so that
\begin{align*}
\Ex\big[\int_0^{T - \Delta t}\int (u(\x,t + \Delta t) - u(\x,t))_+ \;\d\x\d t\big] \leq C_T (\Delta t)^{\beta}
\qquad\mbox{for any $\Delta t\in (0, 1)$}.
\end{align*}
\end{thm}

\begin{pf}
 Define the temporal difference:
\begin{align*}
w(\cdot, t):= u(\cdot, t + \Delta t) - u(\cdot, t).
\end{align*}
From the definition of stochastic kinetic solutions,
for a test function $\varphi(\xi,\x,t)$,
we have
\begin{align*}
& \LL \bar{H}(\xi - u(\cdot,t + \Delta t)), \varphi \RR - \LL \bar{H}(\xi - u(\cdot,t)), \varphi \RR\\
&= \int_t^{t  + \Delta t} \LL \bar{H}(\cdot - u(\cdot,s))\,\FF_u , \nabla \varphi \RR\;\d s
  - \int_t^{t + \Delta t} \LL \, \bar{H}(\cdot - u(\cdot,s))\,D_\x\cdot\FF, \varphi_\xi \RR \;\d s\\
&\quad  + \int_t^{t + \Delta t} \LL \bar{H}(\cdot - u(\cdot,s)), \A: \nabla^2 \varphi\RR \;\d s
  - \int_t^{t+\Delta t}\iint \varphi_\xi\, \d (m^u + n^u)(\xi,\x,s)\\
&\quad + \frac{1}{2} \int_t^{t + \Delta t} \int \sigma^2(u(\x,s)) \varphi_\xi(u(\x,s),\x, t)\;\d\x\d s\\
&\quad + \int_t^{t + \Delta t} \int  \sigma(u(\x,s))\varphi(u(\x,s),\x) \;\d\x\d W(s),
\end{align*}
where, as in (\ref{eq:definsol2}),
the angle brackets represent the integrals in $(\x,\xi)$.
As before, $\bar{H}:= 1 - H$ with $H$ as the Heaviside function.

We now choose a test function that is monotonically increasing in the kinetic variable $\xi$,
so that we can avail ourselves of the sign of the defect measures
in the effort to estimate the left-hand side.
We retain the positive part function
in favor of the sign function.

Nevertheless, inspired by \cite{CDK2012}, we use the test function:
\begin{align*}
\varphi(\xi,\x,t)
= (J_\theta * (\sgn(w(\cdot,t)))_+)(\x)\,\eta_\rho'(\xi - u(\x,t))\ge 0,
\end{align*}
where $J_\theta$ is again an approximation to $\delta_0(\x)$
that is a smooth non-negative function with support on $B_\theta(\mathbf{0})$
and unit mass.
Let $\eta_\rho: \mathbb{R} \to \mathbb{R}$  continue to be
as in the construction given in (\ref{eq:eta_construction}).

Integrating above from $0$ to $T - \Delta t$ in $t$, we have the expression:
\begin{align}
& \int_0^{T - \Delta t}\LL \bar{H}(\xi - u(\cdot,t + \Delta t))- \bar{H}(\xi - u(\cdot,t)),\,\varphi\RR\,\d t\notag\\
&=  \int_0^{T - \Delta t}\int_t^{t  + \Delta t} \LL \bar{H}(\cdot - u(\cdot,s))\,\FF_u, \nabla \varphi \RR\;\d s\d t\notag\\
&\quad  -\int_0^{T - \Delta t} \int_t^{t + \Delta t} \LL \bar{H}(\cdot - u(\cdot,s))\,D_\x\cdot \FF,\,\varphi_\xi\RR \;\d s\d t\notag\\
&\quad + \int_0^{T - \Delta t}\int_t^{t + \Delta t} \LL \bar{H}(\cdot - u(\cdot,s)),\,\A:\nabla^2 \varphi\RR \;\d s\d t\notag\\
&\quad - \int_0^{T - \Delta t}\int_t^{t+\Delta t}\iint \varphi_\xi\,\d (m^u + n^u)(\xi,\x,s)\;\d t\notag\\
&\quad + \frac{1}{2} \int_0^{T - \Delta t} \int_t^{t + \Delta t} \int \sigma^2(u(\x,s)) \varphi_\xi(u(\x,s),\x, t)\;\d\x\d s\d t\notag\\
&\quad + \int_0^{T - \Delta t}\int_t^{t + \Delta t} \int \sigma(u(\x,s))\varphi(u(\x,s),\x, t) \;\d \x\d W(s)\d t\notag\\
&\quad  + \int_0^{T -\Delta t} \LL\bar{H}(\xi - u(\cdot,t + \Delta t)) - |w(\cdot,t)|,\, \varphi \RR\,\d t.\label{8.1a}
\end{align}
Notice that, though the test function $\varphi$ depends on $u(\cdot,t + \Delta t)$,
one can integrate first in $s$ in the stochastic integral above
so that all the integrals are adapted and well-defined either in the Lebesgue--Stieljes sense
or, more generally, in the It\^o sense.

On the left-hand side of \eqref{8.1a}, from the presence of $\eta'_\rho(\xi - u(\cdot,t))$ in the definition
of $\varphi$,  we expect that $\LL H(\xi - u(\cdot,t)), \varphi\RR \to 0$ as $\rho \to 0$. We have the following estimate:
\begin{align*}
\left|\int_0^{T - \Delta t}\LL \bar{H}(\xi - u(\cdot,t)), \varphi\RR\,\d t\right|\leq C_T \rho.
\end{align*}
For the right-hand side of \eqref{8.1a}, we first note  that, as remarked previously,
\[
\int_t^{t+\Delta t}\iint \varphi_\xi\,\d (m^u + n^u)(\xi,\x,s)\geq 0.
\]
We proceed to analyze the remaining parts of the right-hand side of \eqref{8.1a}.

\smallskip
{\it Flux terms}:
Since $D_\x\cdot\FF$ has linear growth in $u$, then the $L^p$ estimate of $u$ implies that
\begin{align*}
&\Big|\Ex\big[\int_0^{T -\Delta t}\int_t^{t  + \Delta t} \LL \bar{H}(\cdot - u(\cdot,s))\,\FF_u, \nabla \varphi\RR\;\d s\d t\big]\\
&\,\, -\Ex\big[\int_0^{T -\Delta t}\int_t^{t + \Delta t} \LL \bar{H}(\cdot - u(\cdot,s))\, D_\x\cdot\FF, \,\varphi_\xi\RR \;\d s\d t\big]\Big|\\
& \leq  C_T \big(\theta^{-1}\Delta t + \rho^{-1}\Delta t\big).
\end{align*}

\smallskip
{\it Parabolic term}:
Using the polynomial growth of the entries of $\A$, we have
\begin{align*}
\Big|\Ex\big[\int_0^{T -\Delta t} \int_t^{t + \Delta t} \LL \bar{H}(\cdot - u(\cdot,s)), \,\A:\nabla^2 \varphi\RR \;\d s\d t\big]\Big|
 \leq  C_T \theta^{-2}\Delta t.
\end{align*}

\smallskip
{\it It\^o Correction term}:
Using the linear growth of $\sigma$,  the $L^p$ estimate of $u$ implies that
\begin{align*}
\frac{1}{2}\Big|\Ex\big[\int_0^{T -\Delta t}\int_t^{t + \Delta t} \int \sigma^2(u(\x,s))\, \varphi_\xi(u(\x,s),\x;t)\;\d\x\d s\d t\big]\Big|
\leq C_T \rho^{-1}\Delta t.
\end{align*}

\smallskip
{\it Noise term}:
Using the Burkholder--Davis--Gundy inequality and the $L^p$ estimate of $u$ yield
\begin{align*}
\Big|\Ex\big[\int_0^{T - \Delta t}\int_t^{t + \Delta t} \int  \sigma(u(\x,s))\varphi(u(\x,s),\x) \;\d\x\d W(s)\d t\big]\Big|
\leq C_T \sqrt{\Delta t}.
\end{align*}

\smallskip
{\it Mollification term}:
Since $D_\x\cdot \FF$ is $\kappa_2$-H\"older in $\x$,
we use the fractional $BV$ estimate in $\x$  in \S 5 to obtain
as in Chen--Ding--Karlsen \cite{CDK2012}:
\begin{align*}
&\Big|\Ex\big[\int_0^{T -\Delta t} \big(\LL \bar{H}(\xi - u(\cdot,t + \Delta t)), \varphi\RR -w(\cdot,t)_+\big)\,\d t\big]\Big|\\
&\leq  \Ex\big[\int_0^{T - \Delta t}\iint J_\theta( \x - \y) |w(\x,t) - w(\y,t)|\;\d\x\d\y\d t\big]\\
&\leq  \Ex\big[\int_0^T \int J(\z)\int |u(\x,t) - u(\x - \theta \z, t)|\;\d\x\d\z\d t\big]\\
&\leq C_T \theta^{\min\{\kappa_1,\kappa_2\}}.
\end{align*}

\smallskip
{\it Conclusion:}
Taking $\rho = \theta^2$ and $\theta = (\Delta t)^\alpha$, we have
\begin{align*}
&\Ex\big[\int_0^{T - \Delta t} \int |w(\x,t)|\;\d\x\d t\big]\\
&\le C\big((\Delta t)^{2\alpha} + (\Delta t)^{1-\alpha} + (\Delta t)^{1-2\alpha} + (\Delta t)^{\frac{1}{2}}
  + (\Delta t)^{\alpha\min\{\kappa_1,\kappa_2\}}\big).
\end{align*}
This allows us to optimize $\alpha$ to conclude that there exists $\beta$ depending on $\kappa_1$ and $\kappa_2$ such that
\begin{align}
\Ex\big[\int_0^{T - \Delta t}\int |w(\x,t)|\;\d\x\d t\big]
\le C |\Delta t|^\beta.
\end{align}
This completes the proof.
\end{pf}

\begin{rem}
In \cite{CDK2012}, it is conjectured that the optimal bound for the first-order conservation law is $(\Delta t)^{\frac{1}{2}}$.
If the $BV$ bound of the solution is in place of a fractional $BV$ bound, the conjecture holds true on the torus for that case.
However,
in the second-order case, the presence of the second derivative provides another power
of $\theta^{-1}$ in the presence of a spatial $BV$ bound, which
leads to a bound $C(\Delta t)^{\beta}$ under the optimization.
\end{rem}

\appendix
\section{Existence of Solutions to the Uniformly Parabolic Equations}\label{sec:lwp_parabolic}

In this appendix, we sketch out the proof for the well-posedness of strong solutions to
the stochastic parabolic equations of form:
\begin{equation}\label{eq:system1a}
\begin{cases}
\pd_t u  - \nabla\cdot \left( \left(\ep \mathbf{I}+  \A(u)\right) \nabla u \right)+ \nabla \cdot \FF(u,\x) =  \sigma(u) \dot{W}
  \qquad\mbox{on $\mathbb{T}^d\times [0,T]$},\\
u|_{t=0} = u^0
\end{cases}
\end{equation}
whose coefficients satisfy the assumptions laid out in \eqref{eq:A_F}--\eqref{eq:A_alpha}.
This is a small extension of \cite{GR2000} or \cite[\S 4]{FN2008} (see also \cite[Ch.~3]{Cho2015}).

From \cite{GR2000}, we know that there is a unique strong solution to
\begin{equation}\label{eq:system1}
\begin{cases}
\pd_t u   =  \nabla\cdot \left( \left(\ep \mathbf{I}+  \A(v)\right) \nabla u \right) - \nabla \cdot \FF(v,\x)  -  \sigma(v) \dot{W},\\
u|_{t=0} = u^0
\end{cases}
\end{equation}
for a fixed $v$, an adapted process in $L^p(\Omega;C([0,T];L^p(\T^d)))$.

We consider the linearization:
\begin{align*}
\pd_t u^n &= \ep \Delta u^n + \nabla \cdot \left(\A(u^{n - 1}) \nabla u^n\right) - \nabla \cdot \FF(u^{n - 1}, x) - \sigma(u^{n - 1}) \dot{W}.
\end{align*}
By the Duhamel formula, we use the following iteration scheme:
\begin{align*}
u^n  &= G_{n - 1}(t )* u^{0} - \int_0^t G_{n - 1}(t - s) * \nabla \cdot \FF(u^{n - 1},\cdot)\;\d s\\
&\quad\,\,- \int_0^t G_{n - 1}(t - s) * \sigma(u^{n - 1}(\cdot,s))\;\d W(s),
\end{align*}
where the convolution is in $\x$ only, and $G_{n - 1}$ is the parametrix of the parabolic equation:
$$
\pd_t u  - \nabla \cdot \left((\ep I + \A(u^{n - 1}))\nabla u\right) = 0.
$$
Then the existence and uniqueness follow by a fixed-point argument, as in \cite[\S 4]{GR2000},
by using the additional direct estimates on the parametrix:
\begin{align}\label{eq:kernel_mod_heat}
\left\|G_{n - 1}(t) \right\|_{L^1}\le \left\|G(t)\right\|_{L^1}, \,\,\,
\left\|\nabla G_{n - 1}(t) \right\|_{L^1}\le \left\|\nabla G(t)\right\|_{L^1}  \quad\,\,
 \mbox{uniformly in $n$},
\end{align}
where $G$ is the standard heat kernel (with $\A \equiv 0$).
These bounds hold because, by scaling $t$, we see that the eigenvalues of the operator
(on the compact domain $\mathbb{T}^d$)
when $\A \equiv 0$  must be larger than those when $\A\ge 0$.

Consider $\FF$ with linear growth in $u$.
Then we can estimate $\norm{u}_{L^p}^p$ by using Young's convolution inequality and Minkowski's inequality
as follows:
\begin{align*}
\norm{u^n}_{L^p} \le& \norm{G_{n - 1}(t)}_{L^1} \norm{u^0}_{L^p}
+ \int_0^t \norm{\nabla G_{n - 1}(t - s)}_{L^1} \norm{\FF(u^{n - 1},\cdot)}_{L^p} \;\d s\\
& + \Big\|\int_0^t G_{n - 1}(t - s) *\sigma(u^{n -1}(\cdot,s))\;\d W\Big\|_{L^p}.
 \end{align*}

By the Burkholder--Davis--Gundy inequality, Minkowski's inequality,  and Jensen's inequality, we see that
\begin{align*}
&\Ex[\sup_{r \in [0,t]}\Big\|\int_0^r G_{n - 1}(r - s) *\sigma(u^{n -1}(\cdot,s))\;\d W\Big\|_{L^p}^p]\\
&\le \Ex [\Big\|\int_0^t \left|G_{n - 1}(t - s) *\sigma(u^{n -1}(\cdot,s))\right|^2 \;\d s\Big\|_{L^{\frac{p}{2}}}^{\frac{p}{2}}]\\
& \le \Ex [\Big|\int_0^t \norm{G_{n - 1}(t - s) *\sigma(u^{n -1}(\cdot,s))}_{L^p}^{2} \,\d s\Big|^{\frac{p}{2}}]\\
& \le t^{\frac{p}{2} - 1} \Ex[\int_0^t \norm{G_{n - 1}}_{L^1}^p \norm{\sigma(u^{n - 1})}_{L^p}^p\,\d s].
 \end{align*}

Next, using the fact that $\nabla_\x \cdot \FF$  (and hence $\FF$) has at most linear growth in $u$,
the assumption that $\sigma(u)$ has at most linear growth in $u$, and \eqref{eq:kernel_mod_heat},
we see that, for $p \ge 2$ and a sufficiently small time $t$, the map: $u^{n - 1} \mapsto u^n$ is
a contraction on $L^p(\Omega; C([0,t]; L^p(\T^d)))$.
The constant is independent of $n$, and the fixed-point argument can be iterated as usual to yield the existence
and uniqueness on the whole time interval $[0,T]$ for any fixed $T>0$.
This shows that a unique strong solution exists for \eqref{eq:system1a} in  $L^p(\Omega ;C([0,T]; L^p(\T^d)))$.

\bigskip
\medskip
\textbf{Acknowledgements}. $\,$
The authors would like to thank the anonymous reviewer for truly helpful suggestions and remarks.

\smallskip
\bigskip

\bibliographystyle{model2-names}
\bibliography{chen-pang_20210922}

\end{document}